\documentclass[11pt, reqno]{amsart}

\usepackage[utf8]{inputenc}
\usepackage[OT2,T1]{fontenc}
\usepackage[english]{babel}
\usepackage{lmodern}

\usepackage{amsmath}
\usepackage{amsfonts}
\usepackage{amssymb}
\usepackage{amsthm}
\usepackage{mathrsfs}
\usepackage[fleqn]{nccmath}

\usepackage[OT2,T1]{fontenc}
\DeclareSymbolFont{cyrletters}{OT2}{wncyr}{m}{n}
\DeclareMathSymbol{\Sha}{\mathalpha}{cyrletters}{"58}

\usepackage{url}
\usepackage[hidelinks]{hyperref}
\hypersetup{colorlinks=true,urlcolor=blue,citecolor=blue,linkcolor=blue}

\usepackage{multirow}

\usepackage[usenames,dvipsnames]{xcolor}
\usepackage{graphicx}
\usepackage{tikz-cd}

\usepackage{etoolbox}

\usepackage{colonequals}

\usepackage{enumerate}

\numberwithin{equation}{section}

\theoremstyle{plain}
\newtheorem{theorem}[equation]{Theorem}

\newtheorem{proposition}[equation]{Proposition}
\newtheorem{lemma}[equation]{Lemma}
\newtheorem{corollary}[equation]{Corollary}

\newtheorem*{MT}{Main Theorem}
\newtheorem*{SC}{Subgroup Existence Criterion}
\newtheorem*{GMT}{Geometric Main Theorem}

\theoremstyle{definition}
\newtheorem{definition}[equation]{Definition}
\newtheorem{algorithm}[equation]{Algorithm}

\theoremstyle{remark}
\newtheorem{remark}[equation]{Remark}
\newtheorem{example}[equation]{Example}

\def\defi{\textsf}
\def\ext{\!\mid\!}

\def\epsilon{\varepsilon}
\def\rho{\varrho}
\def\theta{\vartheta}
\def\phi{\varphi}

\def\Magma{\textsc{Magma}}

\def\loccit{\textit{loc.\ cit.}}

\DeclareMathOperator{\Aut}{Aut}

\DeclareMathOperator{\chr}{char}

\DeclareMathOperator{\Div}{div}

\DeclareMathOperator{\End}{End}

\DeclareMathOperator{\Gal}{Gal}
\DeclareMathOperator{\GL}{GL}

\DeclareMathOperator{\Hom}{Hom}
\DeclareMathOperator{\im}{im}

\DeclareMathOperator{\Jac}{Jac}
\DeclareMathOperator{\Kum}{Kum}

\DeclareMathOperator{\leg}{leg}

\DeclareMathOperator{\Nm}{Nm}

\DeclareMathOperator{\NS}{NS}

\DeclareMathOperator{\Pic}{Pic}

\DeclareMathOperator{\red}{red}

\DeclareMathOperator{\SL}{SL}

\DeclareMathOperator{\Sp}{Sp}

\DeclareMathOperator{\Stab}{Stab}

\DeclareMathOperator{\Sym}{Sym}

\def\C{\mathbb{C}}

\def\F{\mathbb{F}}

\def\P{\mathbb{P}}
\def\Q{\mathbb{Q}}

\def\Z{\mathbb{Z}}

\def\kbar{\overline{k}}

\def\Qbar{\overline{\Q}}

\def\CC{\mathbb{C}}
\def\FF{\mathbb{F}}

\def\PP{\mathbb{P}}
\def\QQ{\mathbb{Q}}

\def\Bc{\mathcal{B}}

\def\Gc{\mathcal{G}}

\def\Pc{\mathcal{P}}
\def\Qc{\mathcal{Q}}
\def\Rc{\mathcal{R}}

\def\Tc{\mathcal{T}}

\newcommand{\dual}{\operatorname{t}}

\usepackage{fullpage}
\usepackage{mathtools}

\begin{document}

\title{Gluing curves of genus 1 and 2 along their 2-torsion}
\date{\today}

\begin{abstract}
  Let $X$ (resp.\ $Y$) be a curve of genus $1$ (resp.\ $2$) over a base field
  $k$ whose characteristic does not equal $2$. We give criteria for the
  existence of a curve $Z$ over $k$ whose Jacobian is up to twist
  $(2,2,2)$-isogenous to the products of the Jacobians of $X$ and $Y$.
  Moreover, we give algorithms to construct the curve $Z$ once equations for
  $X$ and $Y$ are given. The first of these is based on interpolation methods
  involving numerical results over $\C$ that are proved to be correct over
  general fields a posteriori, whereas the second involves the use of
  hyperplane sections of the Kummer variety of $Y$ whose desingularization is
  isomorphic to $X$. As an application, we find a twist of a Jacobian over $\Q$
  that admits a rational $70$-torsion point.
\end{abstract}

\thanks{The first and third author were supported by the
  Juniorprofessuren-Programm ``Endomorphismen algebraischer Kurven'' of the
  Science Ministry of Baden-Württemberg. The second author was supported by the
  Simons Collaboration in Arithmetic Geometry, Number Theory, and Computation
  via Simons Foundation grant 550033. We thank Christophe Ritzenthaler for his
  insightful advice that led to the results in Section \ref{sec:kummer},
  Davide Lombardo for sharing the argument in the proof of Theorem
  \ref{thm:embedding}, and finally the anonymous referee for his or her many
concrete suggestions for improvement.}

\author[Hanselman]{Jeroen Hanselman}
\address{
  Jeroen Hanselman,
  Institut für Reine Mathematik,
  Universität Ulm,
  Helmholtzstrasse 18,
  89081 Ulm,
  Germany
}
\email{jeroen.hanselman@uni-ulm.de}

\author[Schiavone]{Sam Schiavone}
\address{
  Sam Schiavone,
  Massachusetts Institute of Technology,
  Department of Mathematics,
  77 Massachusetts Ave.,
  Bldg. 2-336
  Cambridge,
  MA 02139,
  United States of America
}
\email{sschiavo@mit.edu}

\author[Sijsling]{Jeroen Sijsling}
\address{
  Jeroen Sijsling,
  Institut für Reine Mathematik,
  Universität Ulm,
  Helmholtzstrasse 18,
  89081 Ulm,
  Germany
}
\email{jeroen.sijsling@uni-ulm.de}


\maketitle

\addtocontents{toc}{\protect\setcounter{tocdepth}{1}}
\section*{Introduction}\label{sec:introduction}

One of the most fundamental properties of abelian varieties is their unique decomposition up to isogeny, also called Poincaré's Complete Reducibility Theorem \cite[5.3.7]{birkenhake-lange}: An abelian variety $A$ over a field $k$ is isogenous to a product
\begin{equation}\label{eq:isogdec}
  A \sim B_1^{e_1} \times \dots \times B_n^{e_n}
\end{equation}
where the abelian varieties $B_i$ are simple and pairwise non-isogenous over $k$, and this decomposition is unique in the sense that up to reordering the isogeny classes of the abelian varieties $B_i$ and the corresponding exponents $e_i$ are uniquely determined.

When $A = \Jac (Z)$ is the Jacobian of a curve of small genus, then there exist algorithms after \cite{cmsv-endos} to calculate the decomposition \eqref{eq:isogdec} over the base field $k$ in terms of the Jacobians of curves over small extensions of $k$ whenever possible. The decomposition of the Jacobian of a curve of genus $2$ is also discussed in \cite{kuhn}. Similarly, when $A = \Jac (Z)$ is the Jacobian of a curve of genus $3$ that admits a degree-$2$ map $Z \to X$ to a genus-$1$ curve $X$, then the results in \cite{ritzenthaler-romagny} furnish an explicit description of the complementary part $B$ in the decomposition $A \sim \Jac (X) \times B$ in terms of the Jacobian of a genus-$2$ curve $Y$.

This article aims to develop algorithms for the converse construction, that is, to produce an abelian variety $A$ given factors $B_i$ as in \eqref{eq:isogdec}. When $A = \Jac (Z)$ and $B_i = \Jac (X_i)$, we also call the curve $Z$ a \defi{gluing} of the curves $X_i$.

\subsection*{Previous work}

Gluing elliptic curves $E_1$ and $E_2$ to a genus-$2$ curve $Z$ was first studied in the seminal article \cite{frey-kani} by Frey and Kani, where explicit criteria for the existence of $Z$ given $E_1$ and $E_2$ are given. In fact, \cite{frey-kani} proves a more precise criterion, in that they also fix a degree $n$ and realize $Z$ as a degree-$n$ cover of both $E_1$ and $E_2$. Similarly, Howe--Lepr\'{e}vost--Poonen \cite{howe-leprevost-poonen} use plane quartic curves with defining equation in the Ciani standard form
\begin{equation}\label{eq:ciani}
  Z : a x^4 + b y^4 + c z^4 + d x^2 y^2 + e y^2 z^2 + f z^2 x^2 = 0
\end{equation}
to realize three given elliptic curves $E_1, E_2, E_3$ over $k$ as factors of the Jacobian of $Z$. 

Finally, \cite{cms} realizes the Jacobian of a given genus-2 curve over $\C$ as part of the Jacobian of a smooth plane quartic, using modular methods. In the more recent preprint \cite{cms2} a pencil of plane, bielliptic genus-three (and their unramified double coverings by canonical genus-five curves) is constructed such that the Prym variety of its general member is two-isogenous to the Jacobian of a generic genus-two curve. Like the results in \cite{cms} that are discussed below, those in \cite{cms2} use that $k = \C$ and do not treat the case of an arithmetic base field.

\subsection*{Results in this paper}

This paper considers the problem of gluing two curves of genus $1$ and $2$ to a curve of genus $3$ along their $2$-torsion over a given base field. More precisely, our main theorem, proved in Section \ref{sec:interpolation}, is as follows.

\begin{MT}
  Let $k$ be a field whose characteristic does not equal $2$, and let $X$ and $Y$ be curves of genus $1$ and $2$ over $k$. Then Algorithm \ref{alg:gluek} returns all isomorphism classes of pairs $(Z, \mu)$ over $k$, where $Z$ is a smooth plane quartic curve over $k$ and where $\mu$ is a class in $k^* / (k^*)^2$ with the property that there exists a $k$-rational quotient map
  \begin{equation}\label{eq:ourstuffspec}
    \Jac (X) \times \Jac (Y) \to \mu \ast \Jac (Z)
  \end{equation}
  by a symplectic subgroup of $(\Jac (X) \times \Jac (Y)) [2]$. In particular, we have an isogeny $\Jac (X) \times \Jac (Y) \sim \mu \ast \Jac (Z)$ over $k$. Here $\mu \ast \Jac (Z)$ denotes the quadratic twist of $\Jac (Z)$ with respect to the automorphism $-1$ and the quadratic extension $k (\sqrt{\mu})$ of $k$.
\end{MT}

The alternative Algorithm \ref{alg:AlgGl} in Section \ref{sec:kummer} also determines equations for curves $Z$ gluing $X$ and $Y$, but these may require a further base extension.

Our results function over any base field $k$ of characteristic not equal to $2$, not necessarily algebraically closed, which makes them relevant in the broader arithmetic-geometric context. Moreover, they allow one to specify both curves $X$ and $Y$, whereas for the previous results \cite{cms} over the special base field $k = \C$ only $Y$ could be specified. We mention here that there is also a short and simple construction of $Z$ over a general base field $k$ once only $Y$ is specified: It is given in \cite[\S 2.2]{hanselman-thesis} and involves the parametrization of a certain conic admitting a $k$-rational point. Similarly, the case where the glued curve $Z$ is hyperelliptic will be dealt with in \cite{hanselman-hyp}. This article restricts its consideration to the more involved case where both $X$ and $Y$ are specified and where $Z$ is smooth plane quartic curve.

\subsection*{Applications}

Being able to work with decompositions \eqref{eq:isogdec} over arbitrary base fields, which our Main Theorem contributes to, is of arithmetic importance, for example when describing $L$-functions: If an abelian surface over $\Q$ splits up to isogeny as $A \sim E_1 \times E_2$ over $\Q$, then we have $L (A, s) = L (E_1, s) L (E_2, s)$, and the modular properties of $A$ can be reduced to those of $E_1$ and $E_2$. This description extends to situations where the decomposition requires an extension of the base field. For a systematic exploration this topic for abelian surfaces, we refer to \cite{potmod}.

Another application is in the context of Sato--Tate groups, and more specifically around the recent classification \cite{fks} in genus $3$. While many Sato--Tate groups in \emph{loc.\ cit.}\ can be realized in a trivial way by abelian threefolds of the form $\Jac (X) \times \Jac (Y)$, it remains a challenge to see if the same can be achieved when considering Jacobians $\Jac (Z)$ of curves $Z$ of genus $3$. The current work allows one to approach the latter problem by starting from suitable products $\Jac (X) \times \Jac (Y)$ and attempting to find a corresponding gluing as described in the Main Theorem (and in the next subsection).

Finally, over $k = \C$, the type of decomposition that we reconstruct is important in the context of certain integrable systems, see \cite{enolski}.

\subsection*{Outline}

We now give a more precise description of our methods, as well as of some important intermediate results on the way to the Main Theorem. We will specify our curves $X$ and $Y$ of genus $1$ and $2$ by equations
\begin{equation}\label{eq:pxintro}
  X : y^2 = p_X (x)
\end{equation}
and
\begin{equation}\label{eq:pyintro}
  Y : y^2 = p_Y (x)
\end{equation}
over $k$. While general curves over $k$ of genus $1$ do not allow a defining equation \eqref{eq:pxintro}, we may reduce to this case, since our constructions only involve the Jacobian of $X$, which can always be realized through such an equation.

For an isogeny \eqref{eq:ourstuffspec} to exist, we need a criterion for the existence of a maximal isotropic subgroup $G$ of $(\Jac (X) \times \Jac (Y))[2]$ that is indecomposable (in the sense of not being a product of subgroups of $\Jac (X)$ and $\Jac (Y)$). This is furnished by Theorem \ref{thm:critres}, which is the following:
\begin{SC}
  Let $X$ and $Y$ be as above. There exists an indecomposable maximal isotropic subgroup $G$ of $(\Jac (X) \times \Jac (Y))[2]$ that is defined over $k$ if and only if
  \begin{enumerate}
    \item $p_Y$ admits a quadratic factor $q_Y$ over $k$;
    \item For the complementary factor $r_Y = p_Y / q_Y$ we have that the cubic resolvents $\rho (p_X)$ and $\rho (r_Y)$ have isomorphic splitting fields over $k$.
  \end{enumerate}
\end{SC}
Note that this criterion allows one to work with simpler splitting fields than those defined by the factors $p_X$ and $r_Y$ themselves.

The existence of such a subgroup $G$ does not always guarantee that the corresponding quotient $(\Jac (X) \times \Jac (Y)) / G$ is the twist of the Jacobian of a plane quartic. Indeed, the polarization on said quotient may be decomposable, or give rise to a hyperelliptic curve. Generically, however, a quotient by a non-decomposable group $G$ is isomorphic as a principally polarized abelian variety to a quadratic twist $\mu \ast \Jac (Z)$ of a Jacobian of a plane quartic curve $Z$, as defined in the Main Theorem. Taking a twist is indepensable, as over an arithmetic base field most principally polarized abelian threefolds are not Jacobians --- for a geometric description of this so-called \defi{Serre obstruction}, see for example the main result in \cite{beauville-ritzenthaler} quoted below as Theorem \ref{thm:brtwists}.

In Section \ref{sec:interpolation}, we proceed to find an expression for the curve $Z$ in terms of the data in the Subgroup Existence Criterion whenever possible. The resulting curve $Z$ will admit a homogeneous ternary quartic equation of the form
\begin{equation}
  Z : G (x^2, y, z) = 0
\end{equation}
from which $\Jac (X)$ can be recovered as the Jacobian of the quotient by the involution $(x, y, z) \mapsto (-x, y, z)$. Note that there is no direct map $Z \to Y$ in general, nor need it exist over the algebraic closure $\kbar$. In fact, Proposition \ref{prop:hyp} shows that this can only happen when $Z$ is hyperelliptic, a case that we excluded from consideration. Therefore we cannot use constructions that only involve covers of curves.

Section \ref{sec:interpolation} takes the following indirect route to constructing $Z$. We start by interpolating results over the complex numbers. When appropriately normalized, these yield formulae that can be verified a posteriori to remain valid over any field of characteristic not equal to $2$. Note that evaluating these formulae, as is done in Algorithm \ref{alg:gluek} corresponding to the Main Theorem, once again only requires passing to the common splitting field of the aforementioned cubic resolvents, and is therefore feasible in practice also when the coefficients of the defining equations of $X$ and $Y$ are large. The formulae also yield the twisting scalar $\mu$ mentioned in the Main Theorem. Moreover, since they are obtained in a highly canonical fashion, applying them in concrete cases such as Example \ref{ex:70} yields small defining coefficients for $Z$ without any further simplification being required.

An alternative and more geometric version of the Main Theorem is given in Section \ref{sec:kummer}. It is essentially a geometric inversion of the results in the seminal paper \cite{ritzenthaler-romagny} by Ritzenthaler and Romagny, and constructs $Z$ as a double cover of $X$ obtained by realizing $X$ birationally as a hyperplane section of the Kummer variety of $Y$. This is the content of Theorem \ref{thm:embedding}, which we state here as follows.
\begin{GMT}
  Let $Z$ be a gluing of $X$ and $Y$ as in the Main Theorem. Let $\Kum(Y)^{\dual} = \Jac(Y)^{\dual}/\langle -1\rangle\subset \PP^3_{k}$ be the Kummer surface associated to $\Jac(Y)^{\dual}$. Then over $\kbar$ there exists a commutative diagram
  \begin{equation}
  \begin{tikzcd}
    Z \arrow[swap]{d}{p} \arrow{r}{i_Z} &  \Jac(Y)^{\dual} \arrow{d}{\pi} \\
    X \arrow[swap]{r}{i_X} & \Kum(Y)^{\dual}.
  \end{tikzcd}
  \end{equation}
  where $p:Z\to X$ is a degree-$2$ cover, $\pi:\Jac(Y)^{\dual} \to \Kum(Y)^{\dual}$ is the quotient map, and where $i_Z$ and $i_X$ are rational maps such that $i_X(X) = H \cap \Kum(Y)^{\dual}$ for a plane $H \subset \PP^3_k$ that passes through two singular points of $\Kum(Y)^{\dual}$.
\end{GMT}

Explicitly, an element of the function field of $X$ whose square root gives rise to the double cover $Z \to X$ can be obtained by restricting a Kummer generator of the extension of function fields $k (\Kum (Y)^{\dual}) \subset k (\Jac (Y)^{\dual})$ (which is described in \cite{Mueller}) to the hyperplane section $X$. Various tricks help make this calculation feasible in practice, especially over finite fields, and many interesting phenomena in this geometric Ansatz remain to be explained and generalized.

We give examples of the aforementioned constructions, both over $\Q$ and over finite fields. Moreover, in Section \ref{sec:examples} we use our results to obtain a Jacobian of a plane quartic curve over $\Q$ whose quadratic twist with respect to the automorphism $-1$ and the extension $\Q (\sqrt{5})$ of $\Q$ admits a rational $70$-torsion point. A full implementation of the results in this article is openly available via a full \Magma\ implementation and example suite at \cite{hsv-git}.

\subsection*{Notations and conventions}

Throughout the article, $k$ denotes a fixed base field, whose characteristic we suppose not to equal $2$. Its absolute Galois group is denoted by $\Gamma_k$.

A curve over $k$ is a separated and geometrically integral scheme of dimension $1$ over $k$. Given an affine equation for a curve, we will identify it with the smooth projective curve that has the same function field. The Jacobian of a curve $X$ is denoted by $\Jac (X)$, and its principal polarization, which we consider as an algebraic equivalence class of line bundles on $X$, is denoted by $\lambda_X$.

\tableofcontents

\addtocontents{toc}{\protect\setcounter{tocdepth}{2}}
\section{Criteria for a gluing to exist}\label{sec:criteria}

Let $n \ge 2$ be a positive integer, and let $X$ (resp.\ $Y$) be a smooth curve
of genus $1$ (resp.\ $2$) over the base field $k$, whose characteristic we
suppose not to divide $n$. Let $\pi_X : \Jac (X) \times \Jac (Y) \to \Jac (X)$
and $\pi_Y : \Jac (X) \times \Jac (Y) \to \Jac (Y)$ be the two canonical
projections.

\begin{definition}\label{def:gluing}
  An \defi{(n,n)-gluing} of the curves $X$ and $Y$ (over $k$) is a pair $(Z,
  \phi)$, where $Z$ is a smooth curve over $k$ and where
  \begin{equation}\label{eq:phi}
    \phi : \Jac (X) \times \Jac (Y) \to \Jac (Z)
  \end{equation}
  is an isogeny with the property that $\phi^* (\lambda_Z)$ is algebraically
  equivalent to the $n$-fold $n (\pi_X^* (\lambda_X) \otimes \pi_Y^*
  (\lambda_Y))$ of the product polarization on $\Jac (X) \times \Jac (Y)$.
\end{definition}

Let
\begin{equation}
  T = (\Jac (X) \times \Jac (Y)) [n] = \Jac (X)[n] \times \Jac (Y)[n].
\end{equation}
Let $V = T (\kbar)$, and consider a maximal isotropic subgroup $G$ of $V$. Then
over $\kbar$ we can form the fppf quotient
\begin{equation}
  Q = (\Jac (X) \times \Jac (Y)) / G .
\end{equation}
Let
\begin{equation}
  \pi_Q : \Jac (X) \times \Jac (Y) \to Q
\end{equation}
be the quotient morphism. By \cite[Proposition 11.25]{moonen}, there exists a
unique principal polarization $\lambda_Q$ on $Q$ whose pullback under $\pi_Q$
is algebraically equivalent to $n (\pi_X^* (\lambda_X) \otimes \pi_Y^*
(\lambda_Y))$. Since we have imposed that $\chr (k)$ does not divide $n$, so
that $T$ is étale, and the quotient morphism $Q$ is defined over $k$ if and
only if the subgroup $G$ is, we obtain the following from the arithmetic
version of Torelli's theorem.

\begin{lemma}\label{lem:gluexist}
  Giving a gluing $(Z, \phi)$ of $X$ and $Y$ over $k$ is the same as giving a
  maximal isotropic subgroup $G$ of $V$ with the following properties.
  \begin{enumerate}
    \item $G$ is stable under the action of the absolute Galois group
      $\Gamma_k$.
    \item Moreover, the principally polarized quotient $(Q, \lambda_Q)$
      constructed from $G$ is $k$-isomorphic to the Jacobian of a smooth curve
      over $k$.
  \end{enumerate}
\end{lemma}

\begin{remark}\label{rem:decomp}
  Condition (ii) in Lemma \ref{lem:gluexist} implies Condition (i), but is
  strictly stronger than it. For example, it implies that the group $G$ cannot
  be \defi{decomposable}, that is, a product $G_X \times G_Y$ of maximal
  isotropic subgroups of $\Jac (X) [n]$ and $\Jac (Y) [n]$. Indeed, in this
  case we have
  \begin{equation}
    Q \cong \Jac (X) / G_X \times \Jac (Y) / G_Y
  \end{equation}
  and $\lambda_Q$ is the corresponding product polarization. This precludes the
  existence of an isomorphism $(\Jac (Z), \lambda_Z) \cong (Q, \lambda_Q)$,
  since principal polarizations of Jacobians are indecomposable.

  The demand in Condition (ii) that $(Q, \lambda_Q)$ be a Jacobian over the
  ground field $k$ itself leads to additional arithmetic subtleties, which will
  be discussed in Section \ref{sec:twists}. We will theretofore first
  concentrate on developing criteria that focus on the weaker Condition (i)
  only.
\end{remark}

\subsection{Structure of maximally isotropic subgroups}

In the situation of Definition \ref{def:gluing}, let $n = p$ be a prime number.
This simplifies the description of the maximal isotropic subgroups $G$ of $V =
T (\kbar)$. Let $V_X = \Jac (X) [p] (\kbar)$ and $V_Y = \Jac (Y) [p] (\kbar)$,
so that $V = V_X \times V_Y$. Moreover, let
\begin{equation}
  E = E_X \times E_Y : V \times V \to \F_p
\end{equation}
be the product of the Weil pairings $E_X$ and $E_Y$ on $V_X$ and $V_Y$.
Finally, let $\pi_X : G \to V_X$ and $\pi_Y : G \to V_Y$ be the canonical
projections.

\begin{proposition}\label{prop:indecsurj}
  If $G$ is indecomposable, then $\pi_X$ is surjective and we have
  \begin{equation}
    \dim (\ker (\pi_X)) = \dim (G \cap (0 \times V_Y)) = 1.
  \end{equation}
\end{proposition}

\begin{proof}
  We have $\dim (G) = 3$. If $\dim (\im (\pi_X)) = 0$, then $\ker (\pi_X)$ can
  be identified with an isotropic subgroup of $0 \times V_Y$ of dimension $3$.
  No such subgroups exist since $V_Y$ has dimension $4$. Therefore $\dim (\im
  (\pi_X)) = 1$, which implies that $\im (\pi_X)$ is a symplectic subgroup of
  $V_X$, so that for all $(x_1, y_1)$ and $(x_2, y_2)$ in $G \subset V_X \times
  V_Y$ we have
  \begin{equation}
    0 = E (((x_1, y_1), (x_2, y_2))) =
    E_X ((x_1, x_2)) + E_Y ((y_1, y_2)) = E_Y ((y_1, y_2)) .
  \end{equation}
  This implies that $\im (\pi_Y)$ is an isotropic subgroup of $V_Y$. Since $G
  \subset \im (\pi_X) \times \im (\pi_Y)$ this forces $\dim (\im (\pi_Y)) = 2$
  and $G = \im (\pi_X) \times \im (\pi_Y)$. This is a contradiction with the
  indecomposability of $G$. The second statement of the proposition then
  follows from the dimension formula.
\end{proof}

Now fix an indecomposable maximal isotropic subgroup $G \subset V_X \times
V_Y$. Let $H \subset V_Y$ be the $1$-dimensional $\F_p$-subspace defined by
\begin{equation}
  H = \pi_Y (\ker (\pi_X)) = \pi_Y (G \cap (0 \times V_Y)) \subset V_Y .
\end{equation}
In other words, $H$ is the subgroup of $V_Y$ generated by the second components
of the vectors of the form $(0, w)$ in $G$. The orthogonal complement
$H^{\perp} \subset V_Y$ is of dimension $3$, so that we have
\begin{equation}
  \dim (H^{\perp} / H) = 2 .
\end{equation}
The symplectic pairing $E_Y$ on $V_Y$ induces one on $H^{\perp} / H$, which we
will denote by $E_{\perp}$. There is a multivalued map
\begin{equation}\label{eq:multival}
  V_X \to H^{\perp} / H
\end{equation}
that sends $x \in V_X$ to $\pi_Y (\pi_X^{-1} (x))$. Since $H \times 0 \subset
G$ and $G$ is isotropic, the elements of $\pi_Y (\pi_X^{-1} (x))$ are in
$H^{\perp}$. The map \eqref{eq:multival} factors to a single-valued linear map
\begin{equation}\label{eq:ellG}
  \ell : V_X \to H^{\perp} / H .
\end{equation}

Now for all $(x_1, y_1)$ and $(x_2, y_2)$ in $G$ we have
\begin{equation}
  0 = E (((x_1, y_1), (x_2, y_2)))
  = E_X ((x_1, x_2)) + E_Y ((y_1, y_2)),
\end{equation}
By construction of $\ell$ we have $y_i + H = \ell (x_i) + H$. Therefore the map
$\ell$ is antisymplectic. In particular, $\ell$ is an isomorphism.

Thus an indecomposable maximal isotropic subgroup $G$ gives rise to a subgroup
$H \subset V_Y$ and an anti-symplectic linear isomorphism $\ell_G : V_X \to
H^{\perp} / H$. As is shown in more detail in \cite{hanselman-gluing}, there is
a converse to this result:

\begin{proposition}\label{prop:class}
  Let $(H, \ell)$ be a pair with $H \subset V_Y$ of dimension $1$ and with
  $\ell : V_X \to H^{\perp} / H$ an anti-symplectic isomorphism. Define
  \begin{equation}
    G = \left\{ (x, y) \in V_X \times V_Y : \ell (x) = y + H \right\} .
  \end{equation}
  Then $G \subset V_X \times V_Y$ is indecomposable and maximal isotropic.

  This construction of $G_{\ell}$ from $(H, \ell)$ is inverse to that of $(H,
  \ell)$ from $G$ above and yields a bijective correspondence between
  indecomposable maximal isotropic subgroups $G \subset V_X \times V_Y$ on the
  one hand and the pairs $(H, \ell)$ under consideration on the other.
\end{proposition}

\begin{proof}
  The first part follows by the same methods as above. Note in particular that
  $G$ is indecomposable because its intersection with $0 \times V_Y$ is of
  dimension $1$. The remainder of the statement follows by direct verification.
\end{proof}

\begin{corollary} \label{cor:indecomp}
  There exist exactly $p (p + 1) (p^4 - 1)$ indecomposable maximal isotropic
  subgroups of $V_X \times V_Y$.
\end{corollary}

\begin{proof}
  By Proposition \ref{prop:class}, giving such an indecomposable maximal
  isotropic subgroup is the same as giving a pair $(H, \ell)$. Since $V_Y$ has
  dimension $4$, there are $(p^4 - 1)/(p - 1)$ possible ways to choose $H$.
  Given $H$, there are $\#\Sp (2, \FF_p) = \#\SL_2 (\FF_p) = p (p^2 - 1)$
  choices for the anti-symplectic isomorphism $\ell$. Taking the product yields
  the cardinality in the statement of the corollary.
\end{proof}

\begin{remark}
  The total number of maximal isotropic subgroups of $V_X \times V_Y$
  (including the indecomposable ones) equals $(p^3 + 1)(p^2 + 1)(p + 1)$.
\end{remark}

From now on, we restrict to the case $p = 2$, in which case we simply call a
$(2, 2)$-gluing a \defi{gluing}. The number of indecomposable maximal isotropic
subgroups of $V_X \times V_Y$ then equals $135$, of which $90$ are
indecomposable in our sense. Note that \cite[Lemma 13]{howe-leprevost-poonen}
classifies indecomposable subgroups in another context than ours, namely when
gluing three genus-$1$ curves together; $V_X \times V_Y$ has $54$
indecomposable subgroups in this alternative sense.

\subsection{Interpretation in terms of roots}\label{subsec:roots}

Having taken $p = 2$, the symplectic vector spaces $V_X$ and $V_Y$ admit the
following concrete descriptions. Choose quadratic defining equations
\begin{equation}\label{eq:pX}
  X : y^2 = p_X (x)
\end{equation}
and
\begin{equation}\label{eq:pY}
  Y : y^2 = p_Y (x)
\end{equation}
over $\kbar$, as one may since $\chr (k) \neq 2$. For now, suppose that $p_X$
(resp.\ $p_Y$) is of degree $4$ (resp.\ $6$). Let
\begin{equation}
  \alpha_1, \alpha_2, \alpha_3, \alpha_4 \in \kbar
\end{equation}
be the roots of $p_X$, and let
\begin{equation}
  \beta_1, \beta_2, \beta_3, \beta_4, \beta_5, \beta_6 \in \kbar
\end{equation}
be the roots of $p_Y$. Consider the corresponding sets
\begin{equation}
  \Pc = \left\{ P_i = (\alpha_i, 0) \in X (\kbar) : i \in \left\{ 1, \dots, 4 \right\} \right\}, \qquad
  \Qc = \left\{ Q_i = (\beta_j,  0) \in Y (\kbar) : j \in \left\{ 1, \dots, 6 \right\} \right\} .
\end{equation}

\begin{remark}
  If either one of the degrees of $p_X$ and $p_Y$ is odd, then we can formally
  consider $\infty$ as an element of $\Pc$ or $\Qc$ (or both). To simplify the
  exposition, we ignore these cases; of course our final results still take
  them into account.
\end{remark}

Given a set $T$ of even cardinality, we can define a symplectic $\F_2$-vector
space $\Gc (T)$ as follows: The elements of $\Gc (T)$ are the subsets $S$ of
$T$ of even cardinality up to the equivalence $S \sim S^c$. The symmetric
difference operation
\begin{equation}
  (S_1, S_2) \mapsto S_1 \oplus S_2 = (S_1 \cup S_2) \setminus (S_1 \cap S_2).
\end{equation}
descends to a group structure on $\Gc (T)$, for which the empty subgroup
corresponds to the identity element. Finally, we equip $\Gc (T)$ with the
symplectic pairing
\begin{equation}
  \begin{split}
    \Gc (T) \times \Gc (T) & \to \FF_2 \\
    (S_1, S_2) & \mapsto \# (S_1 \cap S_2) \bmod 2.
  \end{split}
\end{equation}
Now \cite{mumford-tata} shows the following.

\begin{proposition}
  The symplectic $\F_2$-vector space $V_X$ (resp.\ $V_Y$) can be identified
  with the $\Gc (\Pc)$ (resp.\ $\Gc (\Qc)$). To the equivalence class
  $\overline{S}$ of a subgroup $S = \left\{ P_1, P_2 \right\}$ of cardinality
  $2$ there corresponds the $2$-torsion point $[P_1] - [P_2]$.
\end{proposition}

\begin{remark}
  For the genus-$2$ curve $Y$, the subsets of $\Qc$ of cardinality $2$ are in
  bijective correspondence with the non-zero elements of $V_Y$. Indeed, the
  subsets $S$ of $\Pc$ of even cardinality that do not give rise to $0 \in V_Y$
  are of cardinality $2$ or $4$, so that exactly one of $S$ and $S^c$ is of
  cardinality $2$.

  By contrast, for the genus-$1$ curve $X$ the non-zero elements of $V_X$ are
  no longer in bijective correspondence with the subsets $S$ of $\Pc$
  cardinality $2$: this needs the identification of a set $S$ with its
  complement $S^c$.
\end{remark}

Consider a pair $(H, \ell)$ as in Proposition \ref{prop:class}. In terms of
$\Pc$ and $\Qc$, giving $H$ is nothing but giving a subset $\Tc$ of $\Qc$ of
cardinality $2$. Let $\Rc = \Qc \setminus \Tc$ be the complement of $\Tc$.

\begin{proposition}\label{prop:sympiso}
  The inclusion $\iota : \Rc \hookrightarrow \Qc$ induces a canonical
  isomorphism of symplectic $\F_2$-vector spaces
  \begin{equation}\label{eq:RandH}
    i_R : \Gc (\Rc) \to H^{\perp} / H.
  \end{equation}
\end{proposition}

\begin{proof}
  Both vector spaces involved are of dimension $2$. The injection $\iota$ gives
  rise to a well-defined map on equivalence classes $\overline{S}$ since for $S
  \subset \Rc$ we have
  \begin{equation}
    \iota (S^c) = \iota (\Rc \setminus S) = \Qc \setminus (\Tc \cup S)
    \sim \Tc \cup S .
  \end{equation}
  Taking the symmetric difference with the non-trivial element $\Tc$ in $H$, we
  obtain the class
  \begin{equation}
    (\Tc \cup S) \oplus \Tc = S = \iota (S) .
  \end{equation}
  This shows that the images of $\iota (S^c)$ and $\iota (S)$ in $H^{\perp} / H$
  indeed coincide.

  The map $i_R$ is linear (and hence symplectic) since
  \begin{equation}
    \iota (S_1 \oplus S_2) = \iota ((S_1 \cup S_2) \setminus (S_1 \cap S_2)) =
    (S_1 \cup S_2) \setminus (S_1 \cap S_2) = \iota (S_1) \oplus \iota (S_2) .
  \end{equation}
  Finally, $i_R$ is injective since the image of the equivalence class of a
  subset $S \subset R$ of cardinality $2$ remains an equivalence class of a
  subset of cardinality $2$ and is therefore non-trivial. We conclude that
  $i_R$ is indeed a symplectic isomorphism.
\end{proof}

Combining the above results, we get the following.

\begin{corollary}\label{cor:roottrans}
  Giving an indecomposable maximal isotropic subgroup $G \subset V_X \times
  V_Y$ is the same as giving a subset $\Tc$ of $\Qc$ of cardinality 2 along
  with a symplectic isomorphism
  \begin{equation}
    \ell : \Gc (\Pc) \to \Gc (\Rc),
  \end{equation}
  where $\Rc = \Qc \setminus \Tc$.
\end{corollary}

\subsection{Rationality criteria}

Having heretofore worked over the algebraic closure $\kbar$, we now consider
criteria for Condition (i) in Lemma \ref{lem:gluexist}. This requires the
application of Galois theory to detect when a maximal isotropic subgroup $G$ is
defined over the base field $k$.

We start by noting that we may assume $X$ and $Y$ to admit defining equations
\eqref{eq:pX} and \eqref{eq:pY} over $k$. For $Y$ this follows from the fact
that every genus-$2$ curve over $k$ admits such an equation, whereas for $X$ we
may make this assumption because only the Jacobian $\Jac (X)$ intervenes in our
constructions, and this Jacobian is an elliptic curve, which therefore admits
an equation \eqref{eq:pX} since $k$ is not of characteristic $2$.

We proceed to give concrete criteria for the Galois stability in Part (i) of
Lemma \ref{lem:gluexist} in terms of the equivalent interpretation of the
indecomposable maximal isotropic subgroup $G \subset V_X \times V_Y$ that we
developed in the previous section.

\begin{proposition}\label{prop:galstabHell}
  Let $(H, \ell)$ be a pair with $H \subset V_Y$ of dimension $1$ and with
  $\ell : V_X \to H^{\perp} / H$ an anti-symplectic isomorphism, and let $G$ be
  the corresponding indecomposable subgroup of $V$, as described in Proposition
  \ref{prop:class}. Then $G$ is Galois stable if and only if:
  \begin{enumerate}
    \item $H \subset V_Y$ is a Galois stable subspace of dimension $1$;
    \item The map $\ell : V_X \to H^{\perp} / H$ is Galois equivariant.
  \end{enumerate}
\end{proposition}

\begin{proof}
  If $G$ is stable under the action of $\Gamma_k$, then since the same is true
  for $V_Y$, the intersection $H = G \cap (0 \times V_Y)$ is Galois stable. It
  has dimension $1$ by Proposition \ref{prop:indecsurj}, showing that (i) is
  necessary. The same holds for (ii). Indeed, if $G$ is Galois stable, then if
  $(x, y) \in G$, the same holds for $(\sigma (x), \sigma (y))$. This in turn
  means that $\ell (\sigma (x)) = \sigma (y) + H$ in $H^{\perp} / H$, so indeed
  $\ell$ is Galois equivariant, as
  \begin{equation}
    \sigma (\ell (x)) = \sigma (y) + H = \ell (\sigma (x)).
  \end{equation}

  Now suppose conversely that $(H, \ell)$ fulfills the conditions of the
  proposition. Let $(x, y)$ be an element of the corresponding group $G$. We
  have $\ell (x) = y + H$. Since $\ell$ is Galois equivariant, we have
  \begin{equation}
    \ell (\sigma (x)) = \sigma (\ell (x))
                      = \sigma (y) + \sigma (H)
                      = \sigma (y) + H,
  \end{equation}
  which implies $(\sigma (x), \sigma (y)) \in G$. Therefore $G$ is Galois
  stable.
\end{proof}

In terms of our chosen defining defining equations for $X : y^2 = p_X$ and $Y :
y^2 = p_Y$, the Galois stable subgroup $H$ yields a quadratic factor $q_Y$ of
$p_Y$ over $k$, which corresponds to the roots in the set $\Tc$. Let $r_Y = p_Y
/ q_Y$ be the complementary factor corresponding to the roots in $\Rc = \Qc -
\Tc$.

Recall that given a quartic polynomial
\begin{equation}
  p = x^4 + a_3 x^3 + a_2 x^2 + a_1 x + a_0 \in k [x],
\end{equation}
its \defi{cubic resolvent} $\rho (p)$ is defined by
\begin{equation}
  \rho (p) = x^3 - a_2 x^2 + (a_1 a_3 + 4 a_0) x
             + (4 a_0 a_2 - a_1^2 - a_0 a_3^2) \in k[x]
\end{equation}
For simplicity of exposition, we define the cubic resolvent of a general
quartic polynomial as the cubic resolvent of the polynomial obtained by
dividing it by its leading coefficient. The cubic resolvent $\rho (p)$ is known
to be separable if $p$ is, and if $\alpha_1, \dots, \alpha_4$ are the roots of
$p$, then the distinct roots of $\rho (p)$ are given by
\begin{equation}\label{eq:rhoroots}
  \gamma_1 = \alpha_1 \alpha_2 + \alpha_3 \alpha_4,
  \gamma_2 = \alpha_1 \alpha_3 + \alpha_2 \alpha_4,
  \gamma_3 = \alpha_1 \alpha_4 + \alpha_2 \alpha_3.
\end{equation}
While a change of numbering of the roots $\alpha_1, \ldots, \alpha_4$ permutes
the expressions $\gamma_1, \ldots, \gamma_3$, the set $S_p = \left\{ \gamma_1,
\ldots, \gamma_3 \right\}$ does remain invariant. The Galois action on $S_p$
factors through the quotient $\Gal (p)$ of $\Gamma_k$. Under the embedding
$\Gal (p) \hookrightarrow \Sym(4)$ induced by the given numbering $\left\{
\alpha_1, \ldots, \alpha_4 \right\}$, we have
\begin{equation}
  \begin{split}
    \Stab_{\Gal (p)} (\gamma_1) & = \langle (1\, 3\, 2\, 4), (1\, 2) \rangle \cap \Gal (p), \\
    \Stab_{\Gal (p)} (\gamma_2) & = \langle (1\, 2\, 3\, 4), (1\, 3) \rangle \cap \Gal (p), \\
    \Stab_{\Gal (p)} (\gamma_3) & = \langle (1\, 2\, 4\, 3), (1\, 4) \rangle \cap \Gal (p) .
  \end{split}
\end{equation}
The intersection of these stabilizers of $\gamma_i$ corresponds to the subgroup
$V_4 \cap \Gal (p)$ of $\Gal (p)$. Because of its normality, this group is
independent of the chosen numbering of the roots of $p$. We see that the
splitting field of $\rho (p)$ is the number field corresponding to the subgroup
$V_4 \cap \Gal (p) \subset \Gal (p)$ under the Galois correspondence for $\Gal
(p)$. A more precise consideration enables us to prove the following:

\begin{proposition}\label{prop:galvsres}
  There exists a Galois equivariant isomorphism $\ell : \Gc (\Pc) \to \Gc
  (\Rc)$ if and only if the splitting fields of $\rho (p_X)$ and $\rho (r_Y)$
  are isomorphic.
\end{proposition}

\begin{proof}
  Since the trivial element of $\Gc (\Pc)$ is fixed by $\Gamma_k$, its Galois
  structure is determined by the action on the three non-trivial elements of
  $\Gc (\Pc) \setminus \left\{ 0 \right\}$. We have the following isomorphism
  of $\Gamma_k$-sets:
  \begin{equation}\label{eq:galstruct}
    \begin{split}
      \Gc (\Pc) \setminus \left\{ 0 \right\} & \to S_{p_X} \\
      \overline{\left\{ P_1, P_2 \right\}} = \overline{\left\{ P_3, P_4 \right \}}
      & \mapsto \gamma_1 = \alpha_1 \alpha_2 + \alpha_3 \alpha_4, \\
      \overline{\left\{ P_1, P_3 \right\}} = \overline{\left\{ P_2, P_4 \right \}}
      & \mapsto \gamma_2 = \alpha_1 \alpha_3 + \alpha_2 \alpha_4, \\
      \overline{\left\{ P_1, P_4 \right\}} = \overline{\left\{ P_2, P_3 \right \}}
      & \mapsto \gamma_3 = \alpha_1 \alpha_4 + \alpha_2 \alpha_3.
    \end{split}
  \end{equation}
  Indeed, the action factors through the Galois group of the splitting field of
  $p_X$, and the induced action of $\Sym(4)$ on the indices gives rise to an
  identical action on elements on both sides of \eqref{eq:galstruct}.

  It therefore remains to be analyzed when the $\Gamma_k$-sets $S_{p_X}$ and
  $S_{r_Y}$ are isomorphic. Certainly an isomorphism of these sets induces an
  isomorphism of the splitting fields of $p_X$ and $r_Y$. Conversely, if these
  splitting fields are isomorphic, then $\rho_X$ and $\rho_Y$ have a common
  kernel $N$. These representations are therefore conjugate by the following
  lemma, which uses that the polynomials $p_X$ and $r_Y$ are cubic in an
  essential way.
\end{proof}

\begin{lemma}
  Let $\Gamma$ be a group, and let
  \begin{equation}
    \rho_1, \rho_2 : \Gamma \to \Sym(3)
  \end{equation}
  be two representations. Then $\rho_1$ and $\rho_2$ are $\Sym(3)$-conjugate if and
  only if their kernels coincide and their images in $\Sym(3)$ are isomorphic.
\end{lemma}

\begin{proof}
  This is a direct consequence of the fact that two subgroups of $\Sym(3)$ are
  conjugate if and only if they are isomorphic.
\end{proof}

\begin{remark}\label{rem:isonumber}
  It is not much more difficult to calculate the number of Galois equivariant
  isomorphisms $\ell : \Gc (\Pc) \to \Gc (\Rc)$. Indeed, if the given
  $\Gamma_k$-sets are conjugate, then the number of isomorphisms is nothing but
  their common number of automorphisms, which in turn is the number of elements
  of the centralizer of either of their images in $\Sym(3)$.
\end{remark}

\begin{theorem}\label{thm:critres}
  Let $X$ (resp.\ $Y$) be a curve of genus $1$ (resp.\ $2$) admitting a
  defining equation $X : y^2 = p_X$ (resp.\ $Y : y^2 = p_Y$). There exists a
  Galois stable indecomposable maximal isotropic subgroup $G \subset V_X \times
  V_Y$ if and only if
  \begin{enumerate}
    \item $p_Y$ admits a quadratic factor $q_Y$ over $k$;
    \item For the complementary factor $r_Y = p_Y / q_Y$ we have that the
      cubic resolvents $\rho (p_X)$ and $\rho (r_Y)$ have isomorphic splitting
      fields over $k$.
  \end{enumerate}
\end{theorem}

\begin{proof}
  This follows by combining Proposition \ref{prop:galstabHell}(i) and
  Proposition \ref{prop:galvsres}.
\end{proof}

\begin{remark}\label{rem:degen}
  As mentioned at the start of Section \ref{subsec:roots}, some changes take
  place when either $p_X$ of $p_Y$ is of odd degree. If $p_Y$ has degree $5$,
  then we should also consider the case where $p_Y$ has a linear factor over
  $k$ in Part (i) of Theorem \ref{thm:critres}. Similarly, if $p_X$ or $r_Y$ is
  of odd degree, then we should use this polynomial directly in Part (ii)
  instead of taking a Galois resolvent.
\end{remark}

\subsection{Intervening twists}\label{sec:twists}

The previous section has given a concrete characterization of Part (i) of Lemma
\ref{lem:gluexist}. For Part (ii), we restrict ourselves in this article to the
case where the quotient $(\Jac (X) \times \Jac (Y)) / G$ is a Jacobian over
$\kbar$. When $\kbar \subset \C$, it is possible to characterize the case when
this occurs by numerical complex-analytic methods, which are further discussed
in Section \ref{sec:interpolation}: namely, none of the even theta-null values
of the complex torus corresponding to $(\Jac (X) \times \Jac (Y)) / G$ should
vanish.

\begin{remark}
  We note in passing and without detail that it is not difficult to
  characterize when there exists a hyperelliptic gluing $(Z, \phi)$ over
  $\kbar$. This is the case if and only if one of the cross ratios of the roots
  of the polynomial $p_X$ that defines $X$ coincides with one of the cross
  ratios of four of the roots of the polynomial $p_Y$ that defines $Y$, as one
  observes by noting that the Prym variety of the obvious morphism from
  \begin{equation}
    Z : y^2 = x^8 + a_3 x^6 + a_2 x^4 + a_1 x^2 + a_0
  \end{equation}
  to
  \begin{equation}
    X : y^2 = x^4 + a_3 x^3 + a_2 x^2 + a_1 x + a_0
  \end{equation}
  is given by
  \begin{equation}
    Y : y^2 = x (x^4 + a_3 x^3 + a_2 x^2 + a_1 x + a_0).
  \end{equation}

  However, issues of rationality remain to be explored thoroughly. Note in
  particular that it is possible for the glued curve $Z$ to be hyperelliptic
  over $\kbar$ but not over $k$, which happens when $Z$ is a degree-$2$ cover
  of a non-trivial conic. The upcoming work \cite{hanselman-hyp} will explore
  these themes in detail.
\end{remark}

For non-algebraically closed fields, there is an additional complication: It is
possible that $(Z, \lambda_Z)$ is a Jacobian over $\kbar$, but not over $k$.
More precisely, we have the following by \cite{beauville-ritzenthaler}:

\begin{theorem}\label{thm:brtwists}
  Let $(Q, \lambda_Q)$ be a principally polarized abelian threefold over $k$
  that is not a product of abelian varieties of smaller dimension over $\kbar$.
  Then there exists a curve $Z$ over $k$ and a field extension $\ell \ext k$
  with $[\ell : k] \leq 2$ such that $(Q, \lambda_Q)$ is isomorphic to the
  quadratic twist of $(\Jac (Z), \lambda_Z)$ with respect to the automorphism
  $-1$ and the extension $\ell \ext k$. Moreover, if $Z$ is hyperelliptic, then
  $\ell = k$.
\end{theorem}

Given a curve $Z$ over $k$ and a scalar $\mu \in k^*$, we more concisely define
\defi{quadratic twist of $(\Jac (Z), \lambda_Z)$ with respect to $-1$ by $\mu$}
to be the quadratic twist of $(\Jac (Z), \lambda_Z)$ with respect to the
automorphism $-1$ and the extension $k (\sqrt{\mu}) \ext k$. Note that this
twist only depends on the class of $\mu$ in $k^* / (k^*)^2$.

\begin{definition}
  Let $Q = (\Jac (X) \times \Jac (Y)) / G$ and $Z$ be as in the preceding
  theorem, and let $\mu \in k^*$. If $(Q, \lambda_Q)$ is isomorphic to the
  quadratic twist of $(\Jac (Z), \lambda_Z)$ with respect to $-1$ by $\mu$,
  then we call the curve $Z$ a \defi{twisted gluing} of $X$ and $Y$ (with
  \defi{twisting scalar} $\mu$).
\end{definition}

Given generic $X$ and $Y$, the twisting scalar needed to obtain the quotient
$Q$ from $\Jac (Z)$ is usually non-trivial. In other words, over a general
ground field $k$ a twisted gluing is more likely to exist than an actual gluing
(and even the existence of a twisted gluing requires the fairly restrictive
hypotheses of Theorem \ref{thm:critres}). In the plane quartic case that we
will consider in the following section, we will calculate both $Z$ and the
twisting scalar $\mu$ explicitly in terms of the polynomials $p_X$ and $p_Y$
that define the curves $X$ and $Y$ and the cubic resolvents involved in Theorem
\ref{thm:critres}.

\section{Gluing via interpolation}\label{sec:interpolation}

\subsection{Numerical algorithms over $\C$}\label{subsec:interCC}
Consider the base field $k = \C$. In this section, we consider gluings from an
analytic point of view. As in Section \ref{subsec:roots}, we choose defining
equations
\begin{equation}\label{eq:X}
  X : y^2 = p_X (x)
\end{equation}
and
\begin{equation}\label{eq:Y}
  Y : y^2 = p_Y (x) .
\end{equation}
Consider the sets of roots $\Pc = \left\{ \alpha_1, \dots, \alpha_4 \right\}$
of $p_X$ and $\Qc = \left\{ \beta_1, \dots, \beta_6 \right\}$ of $p_Y$. Let
$\Tc = \left\{ \beta_5, \beta_6 \right\}$ and $\Rc = \Qc - \Tc$. Via the
correspondence in Corollary \ref{cor:roottrans}, we consider the maximal
isotropic subgroup $G$ defined by the pair $(\Tc, \ell)$, where $\Tc = \left\{
\beta_5, \beta_6 \right\}$ and where $\ell$ is defined as
\begin{equation}
  \begin{split}
    \ell : \Gc (\Pc) & \to \Gc (\Rc) \\
    \overline{\left\{ \alpha_1, \alpha_i \right\}}
    & \mapsto \overline{\left\{ \beta_1, \beta_i \right\}} .
  \end{split}
\end{equation}
for $i \in \left\{ 2, 3, 4 \right\}$. In other words, we fix a root pairing
determined by our choice of ordering of the roots of $p_X$ and $p_Y$. We intend
to find a corresponding genus-$3$ curve defined by a homogeneous ternary quartic
equation
\begin{equation}\label{eq:Z}
  Z : F_Z (x, y, z) = 0 \subset \P^2
\end{equation}
provided that such an equation exists.

\begin{definition}
  Let $X$ be a curve over $\C$, and let $\Bc = \left\{ \omega_1, \dots,
  \omega_g \right\}$ be a basis of the $\C$-vector space of global
  differentials $H^0 (X, \omega_X)$. The \defi{period lattice} $\Lambda_{X,
  \Bc}$ of $X$ with respect to $\Bc$ is the lattice in $\C^g$ defined by
  \begin{equation}
    \Lambda_{X, \Bc} = \left\{ \left(\int_\gamma \omega_i \right)_{i=1, \ldots, g} :
                               \gamma \in H_1 (X, \Z) \right\} .
  \end{equation}
  It carries a natural symplectic pairing $E_{\Lambda_{X, \Bc}}$ coming from
  the intersection product. We call a pair $(\Lambda_{X, \Bc}, E_{\Lambda_{X,
  \Bc}})$ thus obtained from some curve $X$ of genus $g$ and some basis $\Bc$
  of $H^0 (X, \omega_X)$ a \defi{period lattice with a distinguished symplectic
  pairing} in $\CC^g$. When its presence is clear from the context, we remove
  mention of $E_{\Lambda_{X, \Bc}}$.
\end{definition}

A choice \eqref{eq:X} of defining equation picks out a distinguished basis of
$H^0 (X, \omega_X)$, namely
\begin{equation}\label{eq:BX}
  \Bc_{p_X} = \{ dx / y \}
\end{equation}
Similarly, the defining equations \eqref{eq:Y} and \eqref{eq:Z} pick out the
bases
\begin{equation}\label{eq:BY}
  \Bc_{p_Y} = \{ x dx / y, dx / y \}
\end{equation}
and
\begin{equation}\label{eq:BZ}
  \Bc_{F_Z} = \{ x dx / (\partial f_Z / \partial y),
                 y dx / (\partial f_Z / \partial y),
                   dx / (\partial f_Z / \partial y) \},
\end{equation}
where $f_Z (x, y) = F_Z (x, y, 1)$. For genus-$2$ curves, this yields a map
\begin{equation}\label{eq:mapg2}
  \begin{split}
    \left\{ \begin{array}{c}
      \text{Defining polynomials $p_Y$ in \eqref{eq:Y}}
      \end{array}\right\}
    &\to
    \left\{ \begin{array}{c}
      \text{Period lattices with a distinguished}\\
      \text{symplectic pairing in $\CC^2$}
      \end{array}\right\} \\
    p_Y &\mapsto \Lambda_{Y, \Bc_{p_Y}} .
  \end{split}
\end{equation}
Similarly, for smooth plane quartic curves we obtain a map
\begin{equation}\label{eq:mapg3}
  \begin{split}
    \left\{ \begin{array}{c}
      \text{Defining equations $F_Z$ in \eqref{eq:Z}}
      \end{array}\right\}
    &\to
    \left\{ \begin{array}{c}
      \text{Period lattices with a distinguished}\\
      \text{symplectic pairing in $\CC^3$}
      \end{array}\right\} \\
    F_Z &\mapsto \Lambda_{Z, \Bc_{F_Z}} .
  \end{split}
\end{equation}

\begin{proposition}\label{prop:eqvslat}
  The map \eqref{eq:mapg2} is bijection, and the map \eqref{eq:mapg3} yields a
  bijection after quotienting out the identification $F_Z \sim -F_Z$.
\end{proposition}

\begin{proof}
  We first prove surjectivity of \eqref{eq:mapg2}. By definition, given a
  period lattice with symplectic pairing $\Lambda$, there exists a pair $(Y,
  \Bc)$ that gives rise to it. We may assume that $Y$ is defined by a
  polynomial $p_Y$ as in \eqref{eq:BY}. At this point, the sole remaining
  problem is that the basis $\Bc$ need not be the distinguished basis
  $\Bc_{p_Y}$ associated to $p_Y$.

  Given a matrix $S \in \GL_2 (\CC)$ and a defining polynomial $p$ as in
  \eqref{eq:BY}, we can transform $p$ by the fractional linear transformation
  in $x$ induced by $S^{-1}$ to obtain a new polynomial $p \cdot S^{-1}$ as
  well as a canonical isomorphism from the curve defined by $p$ to the curve
  defined by $p \cdot S^{-1}$. Explicitly, when $S = \begin{pmatrix} a & b \\ c
    & d \end{pmatrix}$, this isomorphism is given by
  \begin{equation}
    (x, y) \mapsto \left(\frac{a x + b}{c x + d}, \frac{y}{c x + d}\right) .
  \end{equation}
  Under the map on differentials induced by this isomorphism, the distinguished
  basis $\Bc_{p \cdot S^{-1}}$ pulls back to a scalar multiple of the (in
  general no longer distinguished) basis $S \cdot \Bc_p$, where the latter
  action of $S$ is defined by treating $\Bc_p$ as a vector.

  Since $\Bc$ and $\Bc_{p_Y}$ are bases of global differentials for the same
  curve $Y$, we have $T \cdot \Bc = \Bc_{p_Y}$ for some $T \in \GL_2 (\CC)$.
  The argument above shows that under the canonical isomorphism of $Y$ with the
  curve defined by $p_Y \cdot T$, the basis $\Bc$ is a scalar multiple of the
  distinguished basis. Replacing $p_Y$ by $p_Y \cdot T$ if necessary, we may
  therefore assume that $(Y, \Bc_{p_Y})$ gives rise to a scalar multiple of
  $\Lambda$. As scaling the polynomial $p_Y$ induces a scaling of the basis of
  differentials $\Bc_Y$, we can therefore indeed find a polynomial giving rise
  to the specified period lattice with symplectic pairing $\Lambda$ under
  \eqref{eq:mapg2}.

  For injectivity we apply Torelli's theorem: Because the symplectic pairing
  that is (implicitly) present on the right-hand side of \eqref{eq:mapg2} is
  fixed, we know that if $p_1$ and $p_2$ on the left-hand side give rise to the
  same period matrix with symplectic pairing, there in fact exists an
  isomorphism between the corresponding hyperelliptic curves. Following the
  argument for surjectivity shows that there even exists an isomorphism whose
  corresponding fractional linear transformation in $x$ is trivial. The
  corresponding map in $y$ is necessarily a scaling by $\lambda \in \CC$, which
  only fixes a given period lattice when $\lambda \in \left\{ \pm 1 \right\}$.
  Since transforming by the hyperelliptic involution does not affect the
  defining equations of hyperelliptic curves, we conclude that $p_1 = p_2$.

  For plane quartics, the same argument can be run for surjectivity, except
  that this time $T \in \GL_3 (\CC)$ acts via fractional linear transformation
  on the projective coordinates $x, y, z$. For injectivity, given two defining
  equations $F_1 = 0$ and $F_2 = 0$ giving rise to the same period matrix with
  symplectic pairing $\Lambda$, we can again conclude the existence of an
  isomorphism between the curves defined by $F_1 = 0$ and $F_2 = 0$ whose
  tangent representation is a scalar. We conclude that $F_2 = \mu F_1$ for a
  scalar $\mu$. As multiplying a ternary quartic $F$ by $\mu$ scales the
  lattice $\Lambda$ by $\mu^{-1}$, we conclude that there remains the ambiguity
  of scaling $F$ by $\mu = -1$. Conversely, no other such scalings fix the
  lattice $\Lambda \subset \CC^3$, so that we obtain a bijection after
  quotienting out the identification in the statement of the proposition.
\end{proof}

\begin{remark}
  An explicit algorithm to recover the polynomial $p_Y$ in \eqref{eq:mapg2}
  that corresponds to a given period lattice with a symplectic pairing is given
  in \cite{guardia}.
\end{remark}

\begin{remark}
  The statement in Proposition \ref{prop:eqvslat} does not hold for defining
  polynomials of genus-$1$ curves $X$. Indeed, the curves
  \begin{equation}
    X_1 : y^2 = x (x - 1) (x - 2)
  \end{equation}
  and
  \begin{equation}
    X_2 : y^2 = (x + 1) x (x - 1)
  \end{equation}
  define the same period lattice, as the isomorphism $(x, y) \mapsto (x - 1,
  y)$ from $X_1$ to $X_2$ preserves the regular differential $d x / y$. Yet the
  defining polynomials of $X_1$ and $X_2$ are not equal.
\end{remark}

Let $\Lambda_X$ and $\Lambda_Y$ be the period lattices resulting from the
choices of basis \eqref{eq:BX} and \eqref{eq:BY}. These matrices can be
calculated especially fast by using new algorithms by Molin--Neurohr
\cite{molin-neurohr}. This functionality also includes the calculation of the
Abel--Jacobi map, which means that we can construct elements of the group $G$,
as follows.

Given a hyperelliptic curve $X$ with Weierstrass points $P_1$ and $P_2$, the
corresponding element of $H^0 (X, \omega_X)^* / H_1 (X, \Z)$ under the
Abel--Jacobi map is
\begin{equation}
  \omega \mapsto \int_{P_1}^{P_2} \omega,
\end{equation}
where the integral can be taken along any path between $P_1$ and $P_2$. Now the
period lattice corresponding to $\Jac (X) \times \Jac (Y)$ is
\begin{equation}
  \Lambda = \Lambda_X \times \Lambda_Y \subset \C^1 \times \C^2 = \C^3 .
\end{equation}
Because of the construction in Proposition \ref{prop:class} and the above
remark, a basis for $G \cong (\Z / 2 \Z)^3$ is given by the elements $\left\{
v_1, v_2, v_3 \right\}$ of
\begin{equation}
  \left\{ \left(0, \int_{\beta_5}^{\beta_6} x dx / y, \int_{\beta_5}^{\beta_6} dx / y\right),
          \left(\int_{\alpha_1}^{\alpha_2} dx / y, \int_{\beta_1}^{\beta_2} x dx / y, \int_{\beta_1}^{\beta_2} dx / y\right),
          \left(\int_{\alpha_1}^{\alpha_3} dx / y, \int_{\beta_1}^{\beta_3} x dx / y, \int_{\beta_1}^{\beta_3} dx / y\right)
        \right\}
\end{equation}
of $\C^3 / \Lambda$. Choosing a basis $\left\{ e_1, \dots, e_6 \right\}$ of the
$\Z$-module $\Lambda$, we can use numerical calculation to find $q_{i,j} \in
(1/2) \Z$ such that
\begin{equation}
  v_i = q_{i,1} e_1 + \dots + q_{i,6} e_6.
\end{equation}
Using linear algebra over $\Z$, we can find a basis for the lattice $\Lambda_Z$
obtained by adjoining the elements of $G$. Moreover, the principal polarization
on $\Lambda$ (which is returned by the algorithms in \cite{molin-neurohr})
extends to a polarization on $\Lambda_Z$, which by construction of $G$ is a
$2$-fold of a principal polarization $E$.

Our task is to find a plane quartic equation \eqref{eq:Z} corresponding to
$(\Lambda_Z, E)$, if it exists. This is accomplished by the following
algorithm, whose implementation builds on the essential algorithms in
\cite{neurohr}.

\begin{algorithm}\label{alg:recCC}
  This algorithm (numerically) reconstructs a ternary quartic from a period
  lattice.

  \emph{Input}: A period lattice $\Lambda$ with a principal polarization $E$.

  \emph{Output}: A ternary plane quartic $F$ whose corresponding period lattice
  with respect to \eqref{eq:BZ} equals $\Lambda$, if it exists.

  \emph{Steps}:
  \begin{enumerate}
    \item Choose a matrix $P = (P_1 \; P_2) \in M_{3,6} (\C)$ with respect to
      a symplectic basis for $E$, and let $\tau = P_2^{-1} P_1$ be a
      corresponding small period matrix.
    \item As in \cite{cmpaper}, check whether $\tau$ has associated vanishing
      even theta-null values. If so, terminate the algorithm.
    \item As in \cite{cmpaper}, construct a Weber model for $\tau$, determine
      the corresponding invariants, and reconstruct a corresponding ternary
      quartic $F$.
    \item Calculate the period lattice $\Lambda_F$ associated to $F$.
    \item Using the methods from \cite[\S 4.1]{bsz}, find a matrix $T$ such
      that $T \Lambda_F = \Lambda R$ for some $R \in \GL_6 (\Z)$. Let $F_0 = F
      \cdot T^{-1}$ be the transform of $F$ by the fractional linear
      transformation in the coordinates induced by $T^{-1}$.
    \item Calculate the period lattice $\Lambda_0$ for $F_0$ and (again using
      \cite{bsz}) find $\mu$ be such that $\mu \Lambda = \Lambda_0$.
    \item Return the ternary quartic $\mu F_0$.
  \end{enumerate}
\end{algorithm}

The correctness of Algorithm \ref{alg:recCC} follows from the sources cited
therein and Proposition \ref{prop:eqvslat}.

\begin{remark}
  The calculation of $\mu$ in Step (vi) is in fact superfluous, as the effect
  of applying $T$ to $G$ on the resulting period lattices can be described in
  terms of a power of its determinant. For ease of exposition, we have used the
  description above.
\end{remark}

\subsection{Interpolation}\label{sec:realdeal}
The results so far are purely numerical and specific to the base field $\C$. We
now interpolate them to obtain explicit formulae. This process leads to very
large formulae that we cannot display in this article, and in fact all of the
considerations in this section will be descriptive rather than explicit.
However, the process of obtaining the relevant formulae is documented in the
ZIP file available at \cite{hsv-zip}. We sketch the main ideas.

We start with formal monic defining equations
\begin{equation}\label{eq:Xformal}
  X : y^2 = (x - \alpha_1) \cdots (x - \alpha_4)
\end{equation}
and
\begin{equation}\label{eq:Yformal}
  Y : y^2 = (x - \beta_1) \cdots (x - \beta_4) (x^2 + a x + b) .
\end{equation}
We can then consider the gluing for the group $G$ specified in the previous
section, with $\beta_5, \beta_6$ the roots of the symmetrized polynomial $x^2 +
a x + b$. This gives rise to a ternary quartic form $F (x, y, z)$ that defines
a curve $Z$. While Proposition \ref{prop:eqvslat} shows that this equation is
only determined up to a minus sign, the corresponding curve is still
canonically determined as a subvariety of $\P^2_\C$. More concretely, we can
obtain a unique normalized defining equation for $Z$ by dividing by the
coefficient of $x^4$ in $F$. We want to determine the dependence of this
equation on the $\alpha_i$ and $\beta_j$, and will achieve this by a suitable
interpolation process that we will prove correct a posteriori in Section
\ref{subsec:verify}.

It is far too ambitious to start working with all $\alpha_i$ and $\beta_j$
simultaneously. Instead, we have gradually worked our way up. We sketch our
procedure.

\begin{enumerate}
  \item First we work with as few moduli parameters as possible and consider
    equations of the form
    \begin{equation}\label{eq:Xspec}
      X : y^2 = x (x - 1) (x - \alpha)
    \end{equation}
    and
    \begin{equation}\label{eq:Yspec}
      Y : y^2 = x (x - 1) (x - \beta) (x^2 + a x + b)
    \end{equation}
    That is, we take $\alpha_1 = \beta_1 = \infty$, $\alpha_2 = \beta_2 = 0$,
    $\alpha_3 = \beta_3 = 1$, $\alpha_4 = \alpha$, $\beta_4 = \beta$, while
    $\beta_5, \beta_6$ are the roots of $x^2 + a x + b$ as before.

    We consider the monomials in $\alpha, \beta, a, b$ of degree at most $4$.
    There are $70$ of these. We therefore generate $200$ quartics over $\Q$ by
    taking random integer values for $\alpha, \beta, a, b$ between $-10$ and
    $10$, and apply the LLL algorithm to the result of Algorithm
    \ref{alg:recCC} in order to obtain good rational approximations numerical
    values for the coefficients of the normalized defining equation of $Z$.
    (Note that we know the resulting equation to be defined over $\Q$, since by
    our choice of $\alpha, \beta, a, b$ the curves $X$ and $Y$ as well as the
    gluing datum are defined over that field.)

    We then try to find rational expressions in $\alpha, \beta, a, b$ for the
    coefficients of $Z$ that interpolate these equations. It turns out that all
    instances are interpolated by the ternary quartic
    \begin{tiny}
      \begin{equation}\label{eq:notasbad}
        \begin{split}
          & (\alpha^2 \beta^2 - \alpha^2 \beta - \alpha \beta^2 + \alpha \beta) x^4 \\
          & + (a \alpha^2 \beta - a \alpha \beta^2 - a \alpha \beta + a \beta^2 + b \alpha^2 - 2 b \alpha \beta + b \beta^2 + \alpha^2 \beta - 2 \alpha \beta^3 + 2 \alpha \beta^2 - 2 \alpha \beta + \beta^3) x^2 y^2 \\
          & + (-2 a \alpha^2 \beta^2 + 4 a \alpha \beta^2 - 2 a \beta^2 - 2 b \alpha^2 \beta + 2 b \alpha \beta^2 + 2 b \alpha \beta - 2 b \beta^2 - 2 \alpha^2 \beta^2 + 2 \alpha \beta^3 + 2 \alpha \beta^2 - 2 \beta^3) x^2 y z \\
          & + (a \alpha^2 \beta^2 - a \alpha \beta^3 - a \alpha \beta^2 + a \beta^3 + b \alpha^2 \beta - 2 b \alpha \beta^3 + 2 b \alpha \beta^2 - 2 b \alpha \beta + b \beta^3 + \alpha^2 \beta^2 - 2 \alpha \beta^3 + \beta^4) x^2 z^2 \\
          & + (-a \alpha \beta^2 + a \alpha \beta + a \beta^3 - a \beta^2 - \alpha \beta^2 + \alpha \beta + \beta^4 - \beta^3) y^4 \\
          & + (2 a \alpha \beta^3 - 2 a \alpha \beta^2 - 2 a \beta^3 + 2 a \beta^2 - 2 b \alpha \beta^2 + 2 b \alpha \beta + 2 b \beta^3 - 2 b \beta^2 + 2 \alpha \beta^3 - 2 \alpha \beta^2 - 2 \beta^4 + 2 \beta^3) y^3 z \\
          & + (a^2 \alpha \beta^3 - a^2 \alpha \beta^2 - a^2 \beta^3 + a^2 \beta^2 + a b \alpha \beta^2 - a b \alpha \beta - a b \beta^3 + a b \beta^2 + a \alpha \beta^3 - a \alpha \beta^2 \\
          & \quad - a \beta^4 + a \beta^3 + 4 b \alpha \beta^3 - 2 b \alpha \beta^2 - 2 b \alpha \beta - 2 b \beta^4 - 2 b \beta^3 + 4 b \beta^2) y^2 z^2 \\
          & + (2 a b \alpha \beta^3 - 2 a b \alpha \beta^2 - 2 a b \beta^3 + 2 a b \beta^2 + 2 b^2 \alpha \beta^2 - 2 b^2 \alpha \beta - 2 b^2 \beta^3 + 2 b^2 \beta^2 - 2 b \alpha \beta^3 + 2 b \alpha \beta^2 + 2 b \beta^4 - 2 b \beta^3) y z^3 \\
          & + (-a b \alpha \beta^3 + a b \alpha \beta^2 + a b \beta^4 - a b \beta^3 - b^2 \alpha \beta^2 + b^2 \alpha \beta + b^2 \beta^4 - b^2 \beta^3) z^4,
        \end{split}
      \end{equation}
    \end{tiny}
    which is still of somewhat acceptable size --- at least after simplifying
    by changing the factor in front of $x^4$ to $\alpha^2 \beta^2 - \alpha^2
    \beta - \alpha \beta^2 + \alpha \beta$. Testing this result on a few
    thousand quartics more confirms it, completing the first step of our
    approach. We also observe that the new factor in front of $z$ that clears
    denominators is nothing but a product of the discriminants of the
    polynomials $x (x - 1) (x - \alpha)$ and $x (x - 1) (x - \beta)$.
  \item The result in (i) already speeds up further considerations, since it
    obviates all but the final two steps in Algorithm \ref{alg:recCC}, saving
    considerable calculation time. We now explore further by keeping one of
    \eqref{eq:Xspec} and \eqref{eq:Yspec} fixed and considering the general
    expression \eqref{eq:Xformal} and \eqref{eq:Yformal} for the other factor.
    At this point, we suspect that the resulting expressions are polynomial
    once the normalized defining equation for $Z$ is multiplied with the
    product of the discriminants of $(x - \alpha_1) \cdots (x - \alpha_4)$ and
    $(x - \beta_1) \cdots (x - \beta_4)$. This turns out to be the case: the
    corresponding interpolation needs a few thousand curves but finds
    corresponding results with very small coefficients, though involving many
    monomials. We jot down the homogeneity degrees of these monomials in the
    $\alpha_i$ and $\beta_j$ for later use.
  \item We now consider the equations \eqref{eq:Xformal} and \eqref{eq:Yformal}
    simultaneously, multiplying the normalized equation for $Z$ with the same
    product of discriminants as before. Knowing what degrees of homogeneity in
    the $\alpha_i$ and $\beta_j$ to expect for every defining coefficient cuts
    the number of candidate monomials down enormously, although it often
    remains considerable, in the order of several thousand at worst. However,
    finding a corresponding number of curves for interpolation is no problem,
    and with enough patience, the corresponding linear-algebraic calculations
    terminate. They again yield formulae with very modest coefficients
    (typically small powers of $2$), but with a very large number of monomials.
  \item Having found these interpolated formulae, we can stress-test them
    further on several thousand more curves, until we are convinced that
    everything checks out.
\end{enumerate}

Section \ref{sec:alt} will mention a theoretical consideration that obviates
the need for the later steps in the interpolation.

\textbf{To summarize the heuristic results so far}: Starting with $X$, $Y$ and
$G$ at the beginning of this section, we have obtained a formula for a plane
quartic
\begin{equation}\label{eq:Zint1}
  Z' : a_{400} x^4 + a_{220} x^2 y^2 + a_{211} x^2 y z + a_{202} x^2 z^2 +
       a_{040} y^4 + a_{031} y^3 z + a_{022} y^2 z^2 + a_{013} y z^3 + a_{004} z^4,
\end{equation}
with the $a_{ijk}$ polynomials in the $\alpha_i$, $\beta_j$, $a$ and $b$ such
that (conjecturally!) for generic values of the parameters, the resulting
substitution yields a plane quartic whose Jacobian is isomorphic over $\C$ to
the quotient $(\Jac (X) \times \Jac (Y)) / G$ with its induced principal
polarization. Note the pleasant form of \eqref{eq:Zint1}, which is a consequence
of its canonicity.

\begin{remark}
  We will not give a precise analysis of the degeneracy locus of the formulae
  obtained above. Since the discriminant of the resulting ternary quartic can
  be verified to be non-zero, it is generically well-defined. Its exact locus
  of definition is best studied in the context of a more detailed theoretical
  approach, perhaps by using the theory of algebraic theta functions, than the
  more ad hoc methods of this article.
\end{remark}

\begin{remark}
  While we have considered the formal equations \eqref{eq:Xformal} and
  \eqref{eq:Yformal} above, one also has to consider the cases where the
  defining polynomial of $X$ and/or $Y$ has odd degree. The resulting
  interpolation procedures are, however, completely similar, so we do not
  consider them further here or in what follows.
\end{remark}

\subsection{Rationality considerations and verification}\label{subsec:verify}

Because we have determined our equation for a quotient by a subgroup of the
$2$-torsion in a canonical way, it stands to reason to expect that the
resulting construction remains valid over base fields whose characteristic does
not equal $2$. This turns out to be the case. First, however, we will discuss
how to prove when the heuristically interpolated equation \eqref{eq:Zint1} is
actually correct, that is, when the Jacobian of the curve $Z'$ in
\eqref{eq:Zint1} actually splits as a product of the Jacobians of the given
curves $X$ and $Y$. We first consider these questions over $\kbar$ and discuss
the base field $k$ in Section \ref{subsec:twists}, where we will also
generalize to the case where $p_X$ and $p_Y$ are not necessarily monic.

\subsubsection{The genus-$1$ factor}\label{sssec:g1}

The curve $Z'$ has an obvious involution $\iota : (x, y, z) \mapsto (-x, y, z)$.
We claim that the Jacobian of the corresponding quotient is generically indeed
$\kbar$-isomorphic to $\Jac (X)$.

\begin{algorithm}\label{alg:g1check}
  This algorithm gives a method to verify the existence of a
  $\kbar$-isomorphism $\Jac (Z' / \iota) \cong \Jac (X)$.

  \emph{Input}: A curve $X$ defined by an equation \eqref{eq:X} and a curve
  $Z'$ defined by an equation \eqref{eq:Zint1}.

  \emph{Output}: A boolean that indicates whether there exists a
  $\kbar$-isomorphism $\Jac (Z' / \iota) \cong \Jac (X)$.

  \emph{Steps}:
  \begin{enumerate}
    \item Write \eqref{eq:Zint1} in the form $A x^4 + B x^2 + C$, with $A, B, C
      \in \kbar [y, z]$.
    \item Let $p = B^2 - 4 A C$, and let $p_0$ be the homogenization of $p_X$.
    \item Let $I$, $J$ (resp.\ $I_0$, $J_0$) be the binary quartic invariants
      of $p$ (resp.\ $p_0$), as defined in \cite{cremona-fisher}.
    \item Check whether we have $(I : J) = (\mu^2 I_0, \mu^3 J_0)$ for
      some $\mu \in \kbar$. If so, return true, and return false otherwise.
  \end{enumerate}
\end{algorithm}

The correctness of the algorithm follows from the observation that the quotient
$Z' / \iota$ is isomorphic to the curve defined by $p$ and the fact that given a
binary quartic $p$ with invariants $I$, $J$ as in \cite{cremona-fisher}, the
Jacobian of the curve corresponding to $p$ is defined by $y^2 = x^3 - 27 I x -
27 J$, see the first footnote in \loccit

\begin{remark}
  The astute reader will note that Algorithm \ref{alg:g1check} does not work in
  characteristic $3$. In this particular case, there are specific additional
  invariants (this time of weight $1$ and $6$ instead of $2$ and $3$) that can
  be used, as is explained and generalized to higher genus in Romain Basson's
  thesis \cite{basson-thesis}. We thank the referee for this reference.
\end{remark}

Algorithm \ref{alg:g1check} is sufficiently simple to be run for the generic
expression \eqref{eq:Zint1} (considered as a ternary quartic over a rational
function field). It yields a positive response. The implementation at
\cite{hsv-git} still performs the corresponding check for every gluing that it
constructs.

\subsubsection{The genus-$2$ factor}

It remains to check whether a complementary factor of the Jacobian of the curve
$Z'$ in \eqref{eq:Zint1} is given by the specified curve $Y$ over $\kbar$. For
this, we use the result \cite{ritzenthaler-romagny} by Ritzenthaler and
Romagny. We summarize their result in the following way:

\begin{theorem}[\cite{ritzenthaler-romagny}]\label{thm:riro}
  Consider a smooth plane quartic curve
  \begin{equation}\label{eq:Zriro}
    Z : x^4 + h (y, z) x^2 + f (y, z) g (y, z) = 0
  \end{equation}
  over $k$, where $h \in k [y, z]$ and $f, g \in \kbar [y, z]$ are binary
  quadratic forms. Let $\iota$ be the involution $\iota : (x, y, z) \mapsto
  (-x, y, z)$. Then there exists a polynomial $p (h, f, g) \in k [x]$ whose
  coefficients are polynomial expressions in those of $h, f, g$ such that the
  Jacobian of the genus-$2$ hyperelliptic curve
  \begin{equation}
    Y : y^2 = p (h, f, g)
  \end{equation}
  defines a degree-$2$ cover of the Prym variety of $Z \to Z / \iota$.
\end{theorem}

\begin{remark}
  As mentioned in \cite{ritzenthaler-romagny}, the formulae from Theorem
  \ref{thm:riro} only apply under certain genericity assumptions. We quietly
  pass over these in what follows, as we do with most results in this section.
\end{remark}

We will not need the exact expression for $p$ except for some formal
transformation properties such as those in the upcoming Lemma. These can be
verified by the reader using the procedures in \cite{ritzenthaler-romagny}. The
corresponding calculations are also checked at \cite{hsv-git}.

\begin{lemma}\label{lem:transfs}
  Denoting the substitution action of $A \in \GL_2 (\kbar)$ on a binary
  quadratic $q$ by $q . A$, we have the following.
  \begin{enumerate}
    \item $p (h, f, g) = p (h, g, f)$.
    \item $p (h, f, g) = p (h, \mu^{-1} f, \mu g)$ for all $\mu \in \kbar^*$.
    \item $p (h . A, f . A, g . A) = p (h, f, g) . A'$, where $A' = U A U^{-1}$
      for $U = \begin{psmallmatrix} -1 & 0 \\ 0 & 1 \end{psmallmatrix}$.
  \end{enumerate}
\end{lemma}

\begin{proposition}\label{prop:krat}
  In the situation of Theorem \ref{thm:riro}, there exists a $k$-rational
  surjective map $\Jac (Z) \to \Jac (Y)$ if and only if the polynomial $p$ has
  coefficients in $k$.
\end{proposition}

\begin{proof}
  The condition is certainly necessary, since $k$-rationality is not a
  meaningful notion otherwise. Suppose from now on that $p$ has coefficients in
  $k$. One way to conclude is to invoke the canonicity of the constructions in
  \cite{ritzenthaler-romagny}, which ensure rationality over the base field
  whenever meaningfully possible. However, there is also a more direct proof.

  First consider the special case where $f = f_0 := y z$. Then $Z$ admits the
  $k$-rational point $P_0 = (0 : 0 : 1)$. To construct a $k$-rational map $\Jac
  (Z) \to \Jac (Y)$, it suffices to construct such a map for divisors of the
  form $[ P ] - [ P_0 ] \in \Jac (Z)$. Moreover, by smoothness of the Jacobian,
  it in turn suffices to indicate the image of a generic point $P$. Given our
  equation for $Y$, we let $\infty$ be the degree-$2$ divisor of $Y$ at
  infinity. We can then specify an element $D - \infty$ of $\Jac (Y)$ using the
  Mumford representation
  \begin{equation}\label{eq:D}
    \begin{split}
      0 & = t^2 + a_1 x + a_2 \\
      y & = b_1 x + b_2
    \end{split}
  \end{equation}
  of a degree $2$ divisor $D$. Since $P$ is a generic point and the map
  involved is non-constant, the coefficients $a_1, a_2, b_1, b_2$ are in the
  function field of $Z$.

  The results at \cite{hsv-git} contain an equation for a non-trivial divisor
  \eqref{eq:D} that was obtained via another interpolation. It is defined over
  the base field. Moreover, it is verified that in terms of the bases
  \eqref{eq:BY} and \eqref{eq:BZ} the pullback of differentials is represented
  by the transpose of the matrix $T_0 = \begin{psmallmatrix} 0 & -2 & 0 \\ 0 &
  0 & 2 \end{psmallmatrix}$.

  The latter property proves that the map $\Jac (Z) \to \Jac (Y)$ is surjective
  and that the corresponding factor of the Jacobian is complementary to the
  quotient by $\iota$. Moreover, it is what we need to extend our results to
  general $Z$ (on which there may not be a $k$-rational point, so that
  describing a map $\Jac (Z) \to \Jac (Y)$ becomes problematic).

  Indeed, now let $F = x^4 + h (y, z) x^2 + f (y, z) g (y, z) \in k [x, y, z]$
  with $f \in \kbar [y, z]$ be general. Then there exists $A \in \GL_2 (\kbar)$
  such that $f = f_0 . A$. Let $F_0 = F . \widetilde{A}^{-1} \in \kbar [x, y,
  z]$, where $\widetilde{A} = \begin{psmallmatrix} 1 & 0 \\ 0 & A
  \end{psmallmatrix}$. Construct the polynomial $p$ (resp.\ $p_0$) and the
  curve $Y$ (resp.\ $Y_0$) corresponding to $F$ (resp.\ $F_0$). The preceding
  special case describes the middle map in
  \begin{equation}\label{eq:tangcomp1}
    \Jac (Z) \to \Jac (Z_0) \to \Jac (Y_0) \to \Jac (Y).
  \end{equation}
  The map $\Jac (Z) \to \Jac (Z_0)$ is induced by the map of curves defined by
  $\widetilde{A}$, and by Lemma \ref{lem:transfs}(iii) there is an isomorphism
  $\Jac (Y_0) \to \Jac (Y)$ induced by $U A^{-1} U^{-1}$. The tangent
  representation of the composition \eqref{eq:tangcomp1} is therefore
  \begin{equation}
    \begin{split}
      U A^{-1} U \cdot T_0 \cdot \widetilde{A}
      & = \begin{pmatrix} -1 & 0 \\ 0 & 1 \end{pmatrix} A^{-1} \begin{pmatrix} -1 & 0 \\ 0 & 1 \end{pmatrix} \begin{pmatrix} 0 & -2 & 0 \\ 0 & 0 & 2 \end{pmatrix} \begin{pmatrix} 1 & 0 \\ 0 & A \end{pmatrix} \\
      & = \begin{pmatrix} -1 & 0 \\ 0 & 1 \end{pmatrix} A^{-1} \begin{pmatrix} 0 & 2 & 0 \\ 0 & 0 & 2 \end{pmatrix} \begin{pmatrix} 1 & 0 \\ 0 & A \end{pmatrix} \\
      & = \begin{pmatrix} -1 & 0 \\ 0 & 1 \end{pmatrix} A^{-1} \begin{pmatrix} 0 & 2 A \end{pmatrix} \\
      & = \begin{pmatrix} -1 & 0 \\ 0 & 1 \end{pmatrix} \begin{pmatrix} 0 & 2 & 0 \\ 0 & 0 & 2 \end{pmatrix} \\
      & = \begin{pmatrix} 0 & -2 & 0 \\ 0 & 0 & 2 \end{pmatrix} = T_0 .
    \end{split}
  \end{equation}
  We conclude that the composition \eqref{eq:tangcomp1} still has tangent
  representation $T_0$. Since the tangent representation uniquely determines
  the corresponding morphism, the latter is defined over $k$ as well.
\end{proof}

Given an equation $Z$ of the form
\begin{equation}\label{eq:Zrirosym}
  Z : x^4 + h (y, z) x^2 + q (y, z) = 0
\end{equation}
there are $3$ choices for the factorization in \eqref{eq:Zriro} over $\kbar$ up
to reordering and scaling. By Lemma \ref{lem:transfs}(i) and (ii), the
resulting equation for the curve $Y$ in Theorem \ref{thm:riro} depends only on
this particular choice of partition of the roots of $q$ into two pairs.

\begin{proposition}
  Given an equation \eqref{eq:Zrirosym} for a plane quartic curve $Z$, there
  exists a choice of factorization \eqref{eq:Zriro} that gives rise to a curve
  $Y$ defined over $k$ under the construction in Theorem \ref{thm:riro} if and
  only if the cubic resolvent $\rho (q)$ of $q$ admits a root over $k$.
\end{proposition}

\begin{proof}
  This follows because the choice of a pair partition of roots of $q$ uniquely
  determines a defining polynomial of $Y$. This polynomial is therefore
  defined over $k$, that is, Galois invariant, if and only if the given
  partition is. This is the case if and only if $\rho (q)$ admits a root.
  Indeed, classical Galois theory shows that the splitting field of $\rho (q)$
  generates the subfield corresponding to the kernel of the conjugation action
  $\Gal (q) \to \Aut (V_4)$ on pairs of roots, and from the description of the
  roots in \eqref{eq:rhoroots} we then see that the stabilizers of an
  individual root of $\rho (q)$ is nothing but the stabilizer of a given pair
  partition.
\end{proof}

Now consider the genus-$2$ curve $Y$ with which we started. The matrix $T_0$
involved in Proposition \ref{prop:krat} shows that if our results are correct,
then we should expect the defining polynomial of $p_Y$ of $Y$ to coincide with
the transformation $p_{Y'} (-x)$ of the polynomial defining one of the
recovered factors $Y'$ from Theorem \ref{thm:riro}. This turns out to be the
case, but only up to a constant. In other words, we will have
\begin{equation}\label{eq:c}
  p_Y (x) = c p_{Y'} (-x)
\end{equation}
for some non-trivial $c \in k$. The reason for this phenomenon is the same as
our reason for only considering isomorphisms over $\kbar$ in Section
\ref{sssec:g1}, namely the presence of twists, a problem to which we turn in
the next section. For now, we record our result over $\kbar$ in the form of a
criterion (whose conditions are of course far stronger than necessary, but are
all that we need in our circumstances):

\begin{algorithm}\label{alg:g2check}
  If the result of this algorithm is positive, then the Jacobian of $Y$ is
  $\kbar$-isogenous to the complement of the image of $\Jac (Z' / \iota)$ in
  $\Jac (Z')$.

  \emph{Input}: A curve $Y$ defined by an equation \eqref{eq:Y} and a curve
  $Z'$ defined by an equation \eqref{eq:Zint1}.

  \emph{Output}: A boolean that, if true, shows that there is a $\kbar$-isogeny
  between $\Jac (Y)$ and the complement of the image of $\Jac (Z' / \iota)$ in
  $\Jac (Z')$.

  \emph{Steps}:
  \begin{enumerate}
    \item Rescale \eqref{eq:Zint1} to the form \eqref{eq:Zriro} and consider
      the three possible factorizations up to scalars $q = f g$ into binary
      quadratic factors.
    \item For each factorization pair, calculate a defining polynomial $p_{Y'}$
      using Theorem \ref{thm:riro}. If for one of these polynomials we have
      $p_Y (x) = c p_{Y'} (-x)$ for some $c \in \kbar$, then return true.
      Otherwise return false.
  \end{enumerate}
\end{algorithm}

In contrast to Algorithm \ref{alg:g1check}, it costs a large amount of time and
memory to run Algorithm \ref{alg:g2check} on the generic expression
\eqref{eq:Zint1}. Still, the implementation at \cite{hsv-git} can rapidly
perform the corresponding check for every concrete gluing that it constructs
before returning its results.

\subsection{Twists}\label{subsec:twists}

Over an algebraically closed field $k$, positive results to the checks in
Algorithms \ref{alg:g1check} and \ref{alg:g2check} previous sections suffice to
demonstrate that for the curve $Z'$ defined in \eqref{eq:Zint1}, the Jacobian
$\Jac (Z')$ is indeed generically isogenous to $\Jac (X) \times \Jac (Y)$. We
are now interested in doing likewise for general base field $k$, for which
twists play a role, as mentioned in Theorem \ref{thm:brtwists}. The main
problems in passing to the base field $k$ are the following:
\begin{enumerate}
  \item The isomorphism $\Jac (Z' / \iota) \to \Jac (X)$ in Algorithm
    \ref{alg:g1check} may be only defined over a proper extension of $k$;
  \item The hyperelliptic curves defined by the polynomials $p_Y$ and $p_{Y'}$
    in Algorithm \ref{alg:g2check} may not be isomorphic over $k$.
\end{enumerate}
In other words, when the necessary hypothesis of Proposition \ref{prop:krat} is
satisfied, we will have
\begin{equation}
  \Jac (Z') \sim \Jac (X') \times \Jac (Y')
\end{equation}
for suitable curves $X'$ and $Y'$ of genus $1$ and $2$ over $k$, but $X'$ and
$Y'$ need not be isomorphic to the specified curves $X$ and $Y$ over the base
field $k$ itself.

We are in Case (i) when $(I : J)$ and $(I_0 : J_0)$ are equivalent in weighted
$(2, 3)$-space over $k$, but not in weighted $(4, 6)$, whereas Case (ii) occurs
when the scalar $c$ in \eqref{eq:c} is not a square in $k$. In either case, the
issue is the presence of various quadratic twists.

We can resolve this problem and get rid of both twists by appropriately
twisting both the curve $Z'$ and its Jacobian. First, given a plane quartic
curve $Z$, its Jacobian $\Jac (Z)$ has the automorphism $-1$, which does not
come from an automorphism $Z$ itself. We write $\mu \ast \Jac (Z)$ to denote
the quadratic twist of $\Jac (Z)$ with respect to $-1$ by $\mu$, as in Theorem
\ref{thm:brtwists}.

Second, there is another twist available. If $Z$ admits a defining equation
\eqref{eq:Zrirosym}, and $\nu \in k^*$, then we can take a quadratic twist of
the curve $Z$ itself with respect to the involution $\iota : (x, y, z) \mapsto
(-x, y, z)$ by the scalar $\nu$. This yields the curve
\begin{equation}\label{eq:Zrirotwist}
  Z_{\nu} : \nu^2 x^4 + h (y, z) \nu x^2 + q (y, z) = 0 .
\end{equation}
Similarly, given a curve $X : y^2 = p_X$ of gonality $2$, we write $X_{\nu} :
y^2 = \nu p_X$ for its quadratic twist with respect to $\iota$ by the scalar
$\nu$.

\begin{lemma}\label{lem:twists}
  Let $Z$ be a genus-$3$ curve defined by an equation \eqref{eq:Zrirosym}.
  Suppose that $X$ and $Y$ are such that $\Jac (Z) \sim \Jac (X) \times \Jac
  (Y)$ over $k$. We have:
  \begin{enumerate}
    \item $\mu \ast \Jac (Z) \sim \Jac (X_{\mu}) \times \Jac (Y_{\mu})$;
    \item $\Jac (Z_{\nu}) \sim \Jac (X) \times \Jac (Y_{\nu})$.
  \end{enumerate}
\end{lemma}

\begin{proof}
  Part (i) follows because the projections $\Jac (Z) \to \Jac (X)$ and $\Jac
  (Z) \to \Jac (Y)$ commute with the automorphism $-1$, which implies that
  over $k$ the twist $\mu \ast \Jac (Z)$ admits the factor $\mu \ast \Jac (X) =
  \Jac (X_{\mu})$.

  Part (ii) can be verified by explicit calculation. Indeed, the quotient
  \begin{equation}
    Z / \iota : x^2 + h (y, z) x + q (y, z) = 0
  \end{equation}
  is clearly $k$-isomorphic to
  \begin{equation}
    Z_{\nu} / \iota : \nu^2 x^2 + h (y, z) \nu x + q (y, z) = 0 .
  \end{equation}
  As for the new $Y_{\mu}$, a direct calculation that its defining polynomial
  is $\nu^{-3}$ times that of $Y$, which implies the claim.
\end{proof}

Since combining the actions in Lemma \ref{lem:twists} allows us to twist both
separate factors $X'$ and $Y'$ in any way desired, we can find $\mu$ and $\nu$
such that $\mu \ast \Jac (Z'_{\nu})$ has the requested factorization $\Jac (X)
\times \Jac (Y)$ up to isogeny over the base field $k$.

This observation allows us equally well to deal with general defining equations
\eqref{eq:X} and \eqref{eq:Y} rather than merely their monic versions
\eqref{eq:Xformal} and \eqref{eq:Yformal} that we used in our interpolation
algorithms up until now. Moreover, the verification Algorithms
\ref{alg:g1check} and \ref{alg:g2check} function equally well over $k$: In
Algorithm \ref{alg:g1check}, it suffices to check for equivalence of $(I : J)$
and $(I_0 : J_0)$ in weighted projective $(4, 6)$-space over $k$ instead of in
weighted projective $(2, 3)$-space over $\kbar$, and in Algorithm
\ref{alg:g2check}, it suffices to demand that $c$ be a square in $k$.

Summarizing all that went before in this section, we have therefore obtained
the following main algorithm, and with it, the Main Theorem:

\begin{algorithm}\label{alg:gluek}
  This algorithm finds gluings of genus-$1$ and genus-$2$ curves along their
  torsion over the base field.

  \emph{Input}: Equations $X : y^2 = p_X$ and $Y : y^2 = p_Y$ that define
  curves of genus $1$ and $2$ over $k$.

  \emph{Output}: A (possibly empty) list $L$ of pairs $(Z, \mu)$, where $Z$ is
  a smooth plane quartic and where $\mu \in k^*$ is a constant such that $\mu
  \ast \Jac (Z) \sim \Jac (X) \times \Jac (Y)$.

  \emph{Steps}:
  \begin{enumerate}
    \item Initialize the empty list $L$. For all quadratic factors $q_Y$ of
      $p_Y$, let $r_Y = p_Y / q_Y$ and perform all next steps but the final
      one.
    \item Check if the splitting fields of $p_X$ and $q_Y$ are isomorphic. If
      so, consider labelings of roots of $p_X$ and $q_Y$ such that the Galois
      actions on the corresponding roots \eqref{eq:rhoroots} of the quadratic
      resolvents are compatible. For all such labelings, perform all next steps but
      the final one.
    \item Construct the interpolated curve $Z'$ as in \eqref{eq:Zint1}. Check
      that the coefficients of $Z'$ belong to $k$. Construct $X' = Z' / \iota$
      as in Algorithm \ref{alg:g1check} and check that $X'$ is a quadratic
      twist of $X$ with respect to $-1$ by $\mu$ say. Check that the cubic
      resolvent of $q = a_{040} y^4 + a_{031} y^3 z + a_{022} y^2 z^2 + a_{013}
      y z^3 + a_{004} z^4$ admits at least one root over $k$, and that for one
      of the roots of this resolvent, we have $p_Y (x) = c p (-x)$ for some $c
      \in k$, where $p'$ is the defining polynomial for the curve $Y'$
      corresponding to the chosen root, as in Algorithm \ref{alg:g2check}.
    \item Let $\nu = c \mu$, and let $Z = Z'_{\nu}$, so that $\mu \ast \Jac (Z)
      = \mu \ast \Jac (Z'_{\nu}) \sim \Jac (X'_{\mu}) \times \Jac (Y'_{\mu
      \nu}) \cong \Jac (X'_{\mu}) \times \Jac (Y'_c) \cong \Jac (X) \times \Jac
      (Y)$ by Lemma \ref{lem:twists}. Append $(Z, \mu)$ to $L$.
    \item Return the list $L$.
  \end{enumerate}
\end{algorithm}

\begin{remark}\label{rem:niceeq}
  The algorithms at \cite{hsv-git} apply a more precise version of the above
  results, which considers the effect of twists on defining equations
  \eqref{eq:X} (and hence on bases of differentials \eqref{eq:BX}) rather than
  on isomorphism classes. We omit the calculations, which only slightly refine
  the twisting scalars $\mu$ and $\nu$ involved, and describe the result.
  Recall from Section \ref{subsec:interCC} that given defining equations for
  the curves $X$, $Y$, and $Z$, we can consider the bases of $\Jac (Z)$ (resp.
  $\Jac (X) \times \Jac (Y)$) that correspond to the bases \eqref{eq:BZ} and
  the union of the pullbacks of \eqref{eq:BX} and \eqref{eq:BY}.

  Now let defining polynomials $p_X$ and $p_Y$ for the input curves $X$ and $Y$
  be given. Then the formulas at \cite{hsv-git} give a ternary quartic equation
  $F$ and a constant $\mu \in k$ such that there exists a map $\Jac (Z) \to
  \Jac (X) \times \Jac (Y)$ with tangent representation $4 \sqrt{\mu}$ with
  respect to the bases corresponding to $p_X$, $p_Y$, and $F$. Thus $\mu \ast
  \Jac (Z)$ is isogenous over $k$ to the product $\Jac (X) \times \Jac (Y)$.
  The factor $4$ is inserted because it makes both $F$ and $\mu$ smaller
  without losing integrality.

  In practice, choosing this completely canonical approach gives rise to very
  agreeable expressions for $Z$ and $\mu$, especially when $p_X$ and $p_Y$ are
  in reduced minimized form. For example, the simple equations in
  \eqref{eq:70}, were found using this method, without any further optimization
  or reduction being needed.
\end{remark}

\subsection{Avoiding the larger interpolation}\label{sec:alt}

Recall from Section \ref{sec:realdeal} that if $X$ and $Y$ are of the special
form
\begin{equation}\label{eq:Xspec2}
  X : y^2 = x (x - 1) (x - \alpha)
\end{equation}
and
\begin{equation}\label{eq:Yspec2}
  Y : y^2 = x (x - 1) (x - \beta) (x^2 + a x + b)
\end{equation}
then the interpolation gives rise to the simpler formula \eqref{eq:notasbad}.
Moreover, in this special case the twisting scalar $\mu$ with respect to $-1$
is given by
\begin{equation}
  \begin{split}
    \mu = \beta (\beta - 1) (\beta - \alpha)(a^2 \alpha^2 \beta^2 - 2 a^2 \alpha \beta^2 + a^2 \beta^2 + 2 a b \alpha^2 \beta - 2 a b \alpha \beta^2 - 2 a b \alpha \beta \\ + 2 a b \beta^2 + 2 a \alpha^2 \beta^2 - 2 a \alpha \beta^3 - 2 a \alpha \beta^2 + 2 a \beta^3 + b^2 \alpha^2 - 2 b^2 \alpha \beta + b^2 \beta^2 \\ + 2 b \alpha^2 \beta - 4 b \alpha \beta^3 + 4 b \alpha \beta^2 - 4 b \alpha \beta + 2 b \beta^3 + \alpha^2 \beta^2 - 2 \alpha \beta^3 + \beta^4) .
  \end{split}
\end{equation}
Since these expressions are so much smaller those involved in \eqref{eq:Zint1},
the verification of their correctness via Algorithms \ref{alg:g1check} and
\ref{alg:g2check} for the former is far less time-intensive than that for the
latter. Moreover, this section will show that we can both reduce the latter
generic check to the former and circumvent all but the first step in the
interpolation in Section \ref{sec:realdeal}. For this, we use the following
lemmata.

\begin{lemma}
  Let $k \subset \C$, and let $X$, $Y$, $Z$ be defined by equations
  \eqref{eq:X}, \eqref{eq:Y}, \eqref{eq:Z}. Let $A \in \GL_2 (k)$, and let $X
  \cdot A$ (resp.\ $Y \cdot A$) be the curve with defining equation $p_X \cdot
  A$ (resp.\ $p_Y \cdot A$). Similarly, let $B \in \GL_3 (k)$, and let $Z \cdot
  B$ be the curve defined by $F \cdot B$. We have the following for the period
  matrices with respect to the canonical bases \eqref{eq:BX}, \eqref{eq:BY},
  and \eqref{eq:BZ}:
  \begin{enumerate}
    \item $\Lambda_{X} = \det (A) \Lambda_{X \cdot A}$;
    \item $\Lambda_{Y} = \det (A) A (\Lambda_{Y \cdot A})$;
    \item $\Lambda_{Z} = \det (B) B (\Lambda_{Z \cdot B})$.
  \end{enumerate}
\end{lemma}

\begin{proof}
  The proof follows by considering the transformation of the bases of
  differentials \eqref{eq:BX}, \eqref{eq:BY}, and \eqref{eq:BZ} under the
  specified substitutions. This follows via direct calculation, but also by
  first observing that the statements hold projectively by our canonical
  choices of differential bases and then using diagonal matrices to determine
  the correct power of the determinant that is needed for actual equality. (We
  need only consider such powers since $\GL_n$ has no other characters.)
\end{proof}

A similar argument shows the following.

\begin{lemma}
  Using the notation from the previous Lemma and Lemma \ref{lem:twists}, we
  have
  \begin{enumerate}
    \item $\Lambda_X = \mu^{1/2} \Lambda_{X_{\mu}}$;
    \item $\Lambda_Y = \mu^{1/2} \Lambda_{Y_{\mu}}$;
    \item $\Lambda_Z = \begin{pmatrix} \mu & 0 & 0 \\ 0 & \mu^{1/2} & 0 \\ 0 &
      0 & \mu^{1/2} \end{pmatrix} \Lambda_{Z_{\mu}}$.
  \end{enumerate}
  Furthermore, when identifying global differential forms of $Z$ with
  translation-invariant differential forms on $\Jac (Z)$, we have that
  $\Lambda_Z = \mu^{1/2} \Lambda_{\mu \ast \Jac (Z)}$.
\end{lemma}

We return to consider the gluing of the generic equations \eqref{eq:Xformal}
and \eqref{eq:Yformal} from Section \ref{sec:realdeal}. There exist rational
expressions $\alpha, \beta$, matrices $A, B$, and scalars $\lambda, \mu$ in
terms of the original formal parameters such that for the corresponding curves
$X' = X (\alpha)$ and $Y' = Y (\beta)$ as in \eqref{eq:Xspec2} and
\eqref{eq:Yspec2} we have
\begin{equation}
  p_X = \lambda (p_{X'} \cdot A),
  \qquad
  p_Y = \mu (p_{Y'} \cdot B).
\end{equation}
Let $Z'$ and the twisting scalar $\nu'$ with respect to $\iota$ describe the
gluing of $X'$ and $Y'$ considered above, and define $Z$ and $\nu$ for $X$ and
$Y$ similarly. Since the gluing data coincide, we have
\begin{equation}
  \begin{split}
    \Lambda_X \times \Lambda_Y & = \Lambda_{\nu \ast \Jac (Z)} R, \\
    \Lambda_{X'} \times \Lambda_{Y'} & = \Lambda_{\nu' \ast \Jac (Z')} R
  \end{split}
\end{equation}
for the same integral matrix $R$. The lemmata above for the hyperelliptic case
show that
\begin{equation}
  \Lambda_{X'} \times \Lambda_{Y'} =
  \begin{pmatrix}
    \lambda^{1/2} \det (A) & 0 \\
    0 & \mu^{1/2} \det (B) B
  \end{pmatrix}
  \Lambda_{X} \times \Lambda_{Y}.
\end{equation}
Therefore the same relation holds between $\Lambda_{\nu' \ast \Jac (Z')}$ and
$\Lambda_{\nu \ast \Jac (Z)}$. Using the lemmata above in the plane quartic
case shows that the gluing of $X$ and $Y$ is described by the curve $Z$ and the
twisting scalar $\nu$ with respect to $\iota$, where
\begin{equation}
  \begin{split}
    \text{$Z = (Z' \cdot T)_{\chi}$, with $T = \begin{pmatrix} 1 & 0 \\ 0 & B
    \end{pmatrix} \in \GL_3 (k)$} \\
    \text{and $\chi$, $\nu$ such that $\chi \nu = \mu$ and $\chi^2 \nu =
    \lambda \det (A)^2 \det (B)^{-2}$}.
  \end{split}
\end{equation}
Since, as in Section \ref{sec:realdeal}, this construction ensures compatibility
with the canonical bases of differentials, the equations from this new approach
are exactly the same as those obtained before. However, this more subtle inroad
has the advantage of proving the correctness of the complicated final
expression along the way from that of the simpler case considered at the
beginning of this section.

Note that, as was mentioned before, for every concrete input, our algorithms at
\cite{hsv-git} still perform a check of correctness regardless before returning
the output, so that incorrect results are never returned.

\subsection{A crucial symmetrization}

A symmetrization of the formulae obtained above that is important in practice
is the following. It is unpleasant, when starting with polynomials $p_X$ and
$p_Y$ defining $X$ and $Y$, to have to determine their roots $\alpha_i$ and
$\beta_j$, as is currently required to apply the first formula
\eqref{eq:Zint1}. This leads to the determination of a compositum of the
splitting field of two quartic polynomials, which already over $\Q$ involves
difficult field arithmetic that cannot be circumvented by reduction algorithms
like \textsc{Pari}'s \texttt{polredabs} since the extensions involved may have
degree larger than $20$.

Instead, we use the more symmetric presentations
\begin{equation}
  p_X = (x - \alpha_1) \cdots (x - \alpha_4)
      = x^4 + a_1 x^3 + a_2 x^2 + a_3 x + a_4
\end{equation}
and
\begin{equation}
  p_Y = (x - \beta_1) \cdots (x - \beta_4) (x^2 + a x + b)
      = (x^4 + b_1 x^3 + b_2 x^2 + b_3 x + b_4) (x^2 + a x + b) .
\end{equation}
The considerations from Section \ref{sec:criteria} show that our chosen gluing
only depends on a pairing of the roots $\gamma_1 = \alpha_1 \alpha_2 + \alpha_3
\alpha_4$, $\gamma_2 = \alpha_1 \alpha_3 + \alpha_2 \alpha_4$, $\gamma_3 =
\alpha_1 \alpha_4 + \alpha_2 \alpha_3$ of the cubic resolvent of $p_X$ with the
roots $\delta_1, \delta_2, \delta_3$ of the cubic resolvent of $x^4 + b_1 x^3 +
b_2 x^2 + b_3 x + b_4$. We therefore expect that the coefficients $a_{ijk}$ in
\eqref{eq:Zint1} are invariant under the Klein Vierergruppen of permutations of
$\alpha_i$ and $\beta_j$ that stabilize all of the $\gamma_i$ and $\delta_j$.
This indeed turns out to be the case.

Invariant theory shows that the polynomial expressions in $\alpha_1, \dots,
\alpha_n$ that are invariant under the distinguished Vierergruppe are
polynomial expressions in the invariants $\gamma_1, \gamma_2, \gamma_3$ of
weight $2$ and the coefficients $a_1$ and $a_3$ of $p_X$, which are of weight
$1$ and $3$, respectively. A similar result of course holds for $\beta_1,
\dots, \beta_4$. We obtain the following result.

\begin{proposition}\label{prop:cubsuff}
  Starting with equations \eqref{eq:Xformal} and \eqref{eq:Yformal}, the
  coefficients of the interpolated polynomial $Z$ in \eqref{eq:Zint1} can be
  expressed as polynomials in the compositum of the splitting fields of the
  cubic resolvents of $p_X$ and $p_Y$. In particular, if these splitting fields
  coincide (a necessary and sufficient condition for a gluing over the base
  field to exist by Theorem \ref{thm:critres}), then explicitly determining the
  coefficients of $Z$ only requires intermediate calculations in this common
  splitting field of a cubic polynomial.
\end{proposition}

\begin{remark}
  Formally writing down the invariant expressions Proposition
  \ref{prop:cubsuff} already cuts down their length by a factor almost $20$.
\end{remark}

In practice, the formulae obtained in this way determine gluings over finite
base fields (not necessarily prime fields) in a fraction of a second, whereas
gluing curves over the rationals whose defining coefficients have about $100$
decimal digits needs a bit over a minute (and results in defining coefficients
with about $1600$ decimal digits). Corresponding test suites are available at
\cite{hsv-git}.

\begin{remark}
  When considering the case where the curve $Y$ admits a defining equation of
  the form $Y : y^2 = p (x^2)$, our algorithms directly recover the Ciani form
  \eqref{eq:ciani}. When $\Jac (Y)$ is itself a $2$-gluing of elliptic curves
  $E_1$ and $E_2$, embedding $X$ into its Jacobian allows use to recover maps
  $X \to E_1$ and $X \to E_2$ of degree $4$ by functoriality.
\end{remark}

\subsection{Examples}\label{sec:examples}

\begin{example}\label{ex:70}
  Consider the genus-$1$ curve defined by the equation
  \begin{equation}
    X : y^2 = 4 x^3 + 5 x^2 - 98 x + 157 = p_X .
  \end{equation}
  It is isomorphic to the elliptic curve with label
  \href{https://www.lmfdb.org/EllipticCurve/Q/118/c/1}{\texttt{118.c1}} in the
  LMFDB \cite{lmfdb}. Similarly, let
  \begin{equation}
    Y : y^2 = x^6 + 2 x^3 - 4 x^2 + 1 = p_Y .
  \end{equation}
  be the genus-$2$ curve isomorphic to the curve in the LMFDB with label
  \href{https://www.lmfdb.org/Genus2Curve/Q/295/a/295/1}{\texttt{295.a.295.1}}.

  First we consider these curves in light of Section \ref{sec:criteria}. The
  defining polynomial $p_Y$ of $Y$ factors as
  \begin{equation}
    (x - 1) (x^2 + x - 1) (x^3 + 2 x + 1) .
  \end{equation}
  We see that there is a unique quadratic factor $q_Y = x^2 + x - 1$, with
  complement $r_Y = x^4 - x^3 + 2 x^2 - x - 1$, and hence a unique subgroup $H
  \subset V_Y$ of dimension $1$ that is Galois-stable. The cubic resolvent of
  $r_Y$ is given by
  \begin{equation}
    \rho (r_Y) = x^3 - 2 x^2 + 5 x - 8 .
  \end{equation}
  This defines the same splitting field as $p_X$: in fact both polynomials
  already define a common number field, which is isomorphic to that defined by
  $x^3 + 2 x - 1$. Note that we have not taken a resolvent of $p_X$, in line
  with Remark \ref{rem:degen}.

  The common splitting field of $p_X$ and $\rho (r_Y)$ has Galois group $S_3$.
  Remark \ref{rem:isonumber} shows that there is a single Galois equivariant
  isomorphism $\ell : \Gc (\Pc) \to \Gc (\Rc)$ for $H$. Since $H$ itself was
  also unique, we conclude that $V_X \times V_Y$ has a single Galois stable
  maximal isotropic subgroup $G$.

  The algorithms of Section \ref{sec:interpolation} show that the quotient $(Q,
  \lambda_Q)$ is a twist with respect to $-1$ by $5$ of the Jacobian $(\Jac
  (Z), \lambda_Z)$ of the plane quartic curve
  \begin{equation}\label{eq:70}
    Z : 32 x^4 + 3 x^2 y^2 - 132 x^2 y z + 37 x^2 z^2
    + 3 y^4 - 14 y^3 z + 7 y^2 z^2 - 6 y z^3 - 2 z^4 = 0.
  \end{equation}
  (More precisely, the twisting scalar $\mu$ with respect to $-1$ from Remark
  \ref{rem:niceeq} is given by $5^3$.)

  The LMFDB tells us that the Jacobian $\Jac (X)$ has a rational $5$-torsion
  point, and that $\Jac (Y)$ has a $14$-torsion point. As the isogeny defined by
  $G$ has degree that is a power of $2$, we can conclude that $(Q, \lambda_Q)$
  has a rational $70$-torsion point if we show that the Galois module
  \begin{equation}
    W = (\Jac (X) [2] (\kbar) \times \Jac (Y) [2] (\kbar)) / G
  \end{equation}
  has a Galois stable subspace of dimension $1$. For this, we use our knowledge
  of the subgroup $G$. The splitting field $L$ of $p_X p_Y$ is of degree $12$ over
  the base field $k = \Q$. We can label the roots $\alpha_1 = \infty$,
  $\alpha_2, \dots, \alpha_4$ of $p_X$ and $\beta_1, \dots, \beta_6$ of $p_Y$
  in such a way that (i) $\beta_1 = 1$, (ii) $\beta_5$ and $\beta_6$ correspond
  to the quadratic factor $q_Y$, and (iii) the Galois action on the $2$-torsion
  points $[\alpha_2] - [\alpha_1]$, $[\alpha_3] - [\alpha_1]$, $[\alpha_4] -
  [\alpha_1]$ in $\Gc (\Pc)$ coincides with that on the elements $[\beta_2] -
  [\beta_1]$, $[\beta_3] - [\beta_1]$, $[\beta_4] - [\beta_1]$ of $\Gc (\Rc)$
  defined by $r_Y$.

  Now consider the bases $\left\{ [\alpha_2] - [\alpha_1], [\alpha_3] -
  [\alpha_1] \right\}$ of $V_X$ and $\left\{ [\beta_2] - [\beta_1], [\beta_3] -
  [\beta_1], [\beta_5] - [\beta_6], [\beta_5 - \beta_1] \right\}$ of $V_Y$. A
  calculation (performed in \cite{hsv-git}) shows that the right action of two
  generators of $\Gal (L \ext k)$ is given by the matrices
  \begin{equation}
    \begin{pmatrix}
      1 & 1 & 0 & 0 & 0 & 0 \\
      1 & 0 & 0 & 0 & 0 & 0 \\
      0 & 0 & 1 & 1 & 1 & 0 \\
      0 & 0 & 1 & 0 & 0 & 0 \\
      0 & 0 & 0 & 0 & 1 & 0 \\
      0 & 0 & 0 & 0 & 1 & 1
    \end{pmatrix},
    \begin{pmatrix}
      1 & 1 & 0 & 0 & 0 & 0 \\
      0 & 1 & 0 & 0 & 0 & 0 \\
      0 & 0 & 1 & 1 & 1 & 0 \\
      0 & 0 & 0 & 1 & 0 & 0 \\
      0 & 0 & 0 & 0 & 1 & 0 \\
      0 & 0 & 0 & 0 & 0 & 1
    \end{pmatrix}.
  \end{equation}
  Because of our ordering of the roots, the subgroup $G$ corresponds to the
  subspace $U$ given by
  \begin{equation}
    U = \langle
    \begin{pmatrix}
      1 & 0 & 1 & 0 & 0 & 0
    \end{pmatrix},
    \begin{pmatrix}
      0 & 1 & 0 & 1 & 0 & 0
    \end{pmatrix},
    \begin{pmatrix}
      0 & 0 & 0 & 0 & 1 & 0
    \end{pmatrix}
    \rangle .
  \end{equation}
  This subspace is indeed stable under the action of the matrices above. If we
  take the images of the standard basis vectors $e_2$, $e_3$, $e_6$ as a basis
  for the corresponding quotient $W = V / U$, then the induced Galois action on
  $W$ is described by the matrices
  \begin{equation}
    \begin{pmatrix}
      1 & 1 & 0 \\
      1 & 0 & 0 \\
      0 & 0 & 1 \\
    \end{pmatrix},
    \begin{pmatrix}
      1 & 1 & 0 \\
      0 & 1 & 0 \\
      0 & 0 & 1 \\
    \end{pmatrix},
  \end{equation}
  There is a single non-trivial vector that is fixed under this action, which
  is the image of $e_6 = [\beta_5] - [\beta_1]$. Indeed the Galois action sends
  $[\beta_5] - [\beta_1]$ either to itself or to $[\beta_6] - [\beta_1]$, which
  is equivalent to $[\beta_5] - [\beta_1]$ modulo $G$, since the latter group
  contains the generator $[\beta_5] - [\beta_6]$ of $H$.

  We have therefore shown that the twist of the Jacobian of the curve
  \eqref{eq:70} with respect to $-1$ by $5$ indeed contains a rational
  $70$-torsion point.
\end{example}

\begin{example}
  The complex-analytic reconstruction techniques in the previous section also
  allow one to construct gluings along $3$-torsion. While it is more difficult
  to find examples of such gluings over the base field, one can still inspect
  which quotients by overlattices have invariants that are numerically in the
  base field. For example, consider the case $k = \Q$ and the elliptic curve
  with LMFDB label
  \href{https://www.lmfdb.org/EllipticCurve/Q/675/d/2}{\texttt{675.d2}}
  \begin{equation}
    X : y^2 = 4 x^3 + 25
  \end{equation}
  together with the genus-$2$ curve
  \begin{equation}
    Y : y^2 = x^5 + 20 x^3 + 36 x
  \end{equation}
  which is a twist of the curve in the LMFDB with label
  \href{https://www.lmfdb.org/Genus2Curve/Q/2916/b/11664/1}{\texttt{2916.b.11664.1}}.
  Over $\Qbar$, the curves $X$ and $Y$ admit gluings along $3$-torsion whose
  invariants are in $\Q$, as is shown in the example files at \cite{hsv-git}.
  Two such gluings are given by the base extensions of
  \begin{equation}
    6 x^4 - 27 x^2 y^2 + 42 x^2 y z + 13 x^2 z^2 - 18 y^4 - 30 y^3 z + 12 y^2
    z^2 +   24 y z^3 + 16 z^4 = 0
  \end{equation}
  and
  \begin{equation}
    14 x^3 z + 3 x^2 y^2 - 210 x y z^2 + 10 y^3 z + 1225 z^4 = 0.
  \end{equation}
  The former of these curves has its full endomorphism ring defined over a
  number field of degree $12$, whereas the latter requires an extension of
  degree $18$. Exact verification of these numerical results above is possible
  via \cite{cmsv-endos}. In fact, these also show that the curve $Y$ is itself
  isogenous to a product of elliptic curves.

  The algorithms at \cite{hsv-git} also allow for the direct gluing of
  threefold products of elliptic curves along $3$-torsion (or $2$-torsion), as
  shown in the example files at \cite{hsv-git}.
\end{example}

\begin{example}
  A final example is given by considering the curves
  \begin{equation}
    X : y^2 = x^4 + 2x^3 + x + 1
  \end{equation}
  and
  \begin{equation}
    Y : y^2 = 2 x^6 + x^4 + x^3 + x^2 + 2 x + 1
  \end{equation}
  over $\F_3$. Our algorithms give two rise to two gluings, defined by the
  equations
  \begin{equation}
    Z_1 : x^4 + x^2 y z + 2 x^2 z^2 + 2 y^3 z + y^2 z^2 + z^4 = 0
  \end{equation}
  and
  \begin{equation}
    Z_2 : x^4 + 2 x^2 y z + x^2 z^2 + 2 y^3 z + y^2 z^2 + z^4 = 0.
  \end{equation}
  For the former, the twisting scalar with respect to $-1$ is trivial, whereas
  the second requires a quadratic twist with respect to $-1$ by the non-square
  $-1 \in \F_3$ to recover the relevant abelian quotient variety.
\end{example}

\section{Gluing via the Kummer variety}\label{sec:kummer}

In this section we describe a geometric algorithm that allows us to construct gluings of the curves $X$ and $Y$ over any algebraically closed base field $k$. The algorithm we describe (Algorithm \ref{alg:AlgGl}) reverses the construction of Ritzenthaler and Romagny in \cite{ritzenthaler-romagny}, mentioned above in Theorem \ref{thm:riro}.

\subsection{Non-hyperelliptic curves of genus 3 that are gluings}
We start by showing that every non-hyperelliptic curve of genus 3 can be given in the form of Theorem \ref{thm:riro}.
\begin{lemma}\label{lem:gon}
  Let $(Z, \phi)$ be a gluing of $X$ and $Y$ over $k$. Then there is a
  degree-$2$ map $Z \to X$ over $k$.
\end{lemma}

\begin{proof}
  Dualizing the gluing map $\phi$ gives a map $\phi^{\dual} : \Jac (Z) \to \Jac (X) \times \Jac (Y)$. Choosing base points on $Z$ and $X$, we obtain an inclusion $Z \to \Jac (Z)$ and an identification $\Jac (X) \cong X$. We can compose to obtain a non-constant map $f : Z \to X$. The corresponding pullback map $f^* : \Jac (X) \to \Jac (Z)$ is obtained by composing the canonical inclusion $\Jac (X) \hookrightarrow \Jac (X) \times \Jac (Y)$ with $\phi$. Because of the defining property of $\phi$, the principal polarization $\lambda_Z$ satisfies $(f^*)^* (\lambda_Z) \equiv 2 \lambda_X$, which implies that $f$ is of degree $2$ by \cite[Lemma 12.3.1]{birkenhake-lange}.
\end{proof}

\begin{proposition} \label{prop:hyp}
Suppose that $Z$ is a gluing of $X$ and $Y$. If there exists a degree $2$ morphism $\pi_2 : Z \to Y$, then $Z$ is hyperelliptic.
\end{proposition}

\begin{proof}
  There exists a degree $2$ morphism $\pi_1: Z \to X$ by Proposition \ref{lem:gon}. Both $\pi_1$ and $\pi_2$ induce involutions of $Z$, which we denote $i_1$ and $i_2$, respectively, and which in turn induce automorphisms $i_1^*$ and $i_2^*$ on $\Jac(Z)$. Note that $i_1$ fixes divisor classes pulled back from $X$ and $i_2$ fixes divisor classes pulled back from $Y$. Thus, by making a suitable choice of basis, we can represent $i_1^*$ and $i_2^*$ by the matrices
  \begin{equation}
    [i_1^*] =
    \begin{bmatrix}
      1 & 0 & 0\\
      0 & -1 & 0\\
      0 & 0 & -1
    \end{bmatrix},
    \qquad
    [i_2^*] =
    \begin{bmatrix}
      -1 & 0 & 0\\
      0 & 1 & 0\\
      0 & 0 & 1
    \end{bmatrix},
  \label{eqn:matrices}
  \end{equation}
  in $\End^0_k(\Jac(Z)) \cong \End^0_k(\Jac(X)) \times \End^0_k(\Jac(Y))$. Letting $i = i_1 \circ i_2$, then
  \begin{equation}
    [i^*] = [i_1^* \circ i_2^*] = [i_1^*] [i_2^*] =
    \begin{bmatrix}
      -1 & 0 & 0\\
      0 & -1 & 0\\
      0 & 0 & -1
    \end{bmatrix} \, .
    \label{eqn:matrix_mult}
  \end{equation}
  Thus $i$ induces the negation map $-1$ on $\Jac(Z)$, so $Z$ is hyperelliptic by \cite[Appendice, Th\'{e}or\`{e}me 3]{Lauter}.
\end{proof}

\begin{proposition}\label{prop:RRIsGlue}
With notation as in Theorem \ref{thm:riro}, the curve $Z$ is a gluing of the curves $Y$ and $X$.
\end{proposition}
\begin{proof}
  Consider the degree-2 cover $p: Z\to X$ given in Theorem \ref{thm:riro}, where $X = Z/\iota$. The map $p$ induces an inclusion $p^*: (\Jac(X), \lambda_X)\to (\Jac(Z), \lambda_Z)$ of polarized abelian varieties, and by \cite[Lemma 12.3.1]{birkenhake-lange} we find that
\begin{equation}
  (p^*)^*\lambda_{Z} = 2\lambda_{X},
\end{equation}
the pullback of $p^*$.

Now $\Jac(Y)$ is isomorphic to the Prym of an allowable singular cover $p':Z'\to X'$ whose normalization is equal to $p:Z\to X$ as is shown in the proof of Theorem 1.1 in \cite{ritzenthaler-romagny}.

By \cite[Theorem 3.7]{beauville-prym}, the principal polarization on the generalized Jacobian $\Jac(\widetilde{Z})$ restricts to $2\lambda_{Y}$ on $\Pr(Z'/X')$ where $\lambda_{Y}$ is the principal polarization on $\Pr(Z' /X')\cong \Jac(Y)$. From \cite[Lemma 1]{donagi-livne} we get a commutative diagram of polarized abelian varieties:

\begin{equation}
  \begin{tikzcd}
    \Pr(Z'/X') \arrow{d}\arrow{r}{\nu} & \Pr(Z/X)  \arrow{d}{i} \\
    \Jac(Z') \arrow{r} & \Jac(Z)
  \end{tikzcd}
\end{equation}
where $\nu$ is induced by the normalization $Z \to Z'$. This implies that
\begin{equation}
  (i\circ\nu)^*(\lambda_{Z}) = 2\lambda_{Y}.
\end{equation}

Now consider the map $j:\Jac(X)\times \Jac(Y)\to \Jac(Z)$ defined by $(x,y)\mapsto p^*(x)+(i\circ\nu)(y)$. As $j =(i\circ\nu)\circ\pi_Y$ on the restriction to $\{0\}\times \Jac(Y)$ we get that
\begin{equation}
  j^*(\lambda_Z)|_{\{0\}\times \Jac(Y)} = (\pi^*_Y\circ(i\circ \nu)^*(\lambda_Z))|_{\{0\}\times \Jac(Y)} = \pi_Y^*(2\lambda_Y)|_{\{0\}\times \Jac(Y)}
\end{equation}
An analogous argument shows that
\begin{equation}
  j^*(\lambda_Z)|_{\Jac(X)\times\{0\}} = (\pi^*_X\circ(p^*)^*(\lambda_Z))|_{\Jac(X)\times \{0\}} =  \pi_X^*(2\lambda_X)|_{\Jac(X)\times \{0\}}.
\end{equation}
Using that our construction is generic, and that generically $\Hom (A, B) = 0$ so that $\NS(A)\times \NS(B) \cong \NS(A\times B)$, we can argue as in \cite[Proposition 2.2]{bruin} to conclude that $j^*(\lambda_Z)$ is algebraically equivalent to $2\pi_X^*(\lambda_{X})\otimes 2\pi_Y^*(\lambda_{Y})$, so $Z$ is a gluing of $\Jac(X)$ and $\Jac(Y)$.

A more concrete approach in the case $k=\C$, which generalizes to arbitrary fields by the use of étale cohomology, is the following. The map $p^*$ gives us an inclusion of $L_X = H_1(X,\Z)$ into $L_Z = H_1(Z,\Z)$ and the equality $(p^*)^*\lambda_{Z} = 2\lambda_{X}$ shows that the restriction of $\lambda_Z$ to $L_X$ gives us $2\lambda_X$. Now the kernel of the map $p_*$ gives us the Prym variety $\Pr(Z/X)$, which, using Lemma 1 of \cite{donagi-livne}, comes equipped with a $(2,1)$-polarization that is the restriction of $\lambda_Z$ to $\Pr(Z/X)$. Let $L_P$ be the sublattice of $L_Z$ that corresponds to $\Pr(Z/X)$. Then $p^*(L_X)\oplus L_P\subset L_Z$. The construction in \cite[Theorem 4.2]{enolski} of which \cite{ritzenthaler-romagny} is the algebraic version, shows that the curve $Y$ corresponds to a sublattice $q^* (L_Y)$ of index $2$ in $L_P$ on which the $(2,1)$-polarization restricts to a $(2,2)$-polarization. As a consequence, the restriction of $\lambda_Z$ to $p^*(L_X)\oplus q^*(L_Y)$ is twice the product polarization, which ensures that the induced map $j:\Jac(X)\times\Jac(Y)\to\Jac(Z)$ has the property that $j^*(\lambda_Z) = 2\pi_X^*(\lambda_X)\otimes 2\pi_Y^*(\lambda_Y)$ on $\Jac(Z)$, which is what we wanted to show.
\end{proof}

\subsection{Plane sections of the Kummer}

We begin by recalling some basic facts about Kummer surfaces. Throughout this section let $X, Y,$ and $Z$ be curves over $k$ of genus $1, 2$, and $3$, respectively.

\begin{definition} \label{def:kummer}
Let $Y$ be a curve of genus $2$ over $k$. The \defi{Kummer surface} of $Y$ is the quotient of $\Jac(Y)$ by the negation map, i.e., $\Kum(Y) = \Jac(Y)/\langle -1 \rangle$.
\end{definition}

\begin{remark}
  \label{rem:kummer_dual}
  Here and throughout we will abuse notation and denote by $\Kum(Y)^{\dual}$ the Kummer surface of $\Jac(Y)^{\dual}$, i.e., $\Kum(Y)^{\dual} = \Jac(Y)^{\dual}/\langle -1 \rangle$.
\end{remark}
\begin{proposition}[{\cite[Proposition 2.16]{Gonzalez-Dorrego}}]
  Let $Y$ be a genus $2$ curve. Then $\Kum(Y)$ has $16$ singular points, each one a node, and there exist $16$ planes, called \defi{special planes}, such that these planes and nodes form a nondegenerate $(16,6)$-configuration.
\end{proposition}

\begin{theorem}
\label{thm:embedding}
Let $Z$ be a non-hyperelliptic curve of genus 3 and let $X$ be a curve of genus 1 such that we have a degree 2 cover $p:Z\to X$ which is ramified above exactly four points as in Theorem \ref{thm:riro}. Let $Y$ be the curve of genus 2 given in this same theorem. Then there exist maps
\begin{equation}
  i_Z:Z\to \Jac(Y)^{\dual} \text{ and } i_X:X\to \Kum(Y)^{\dual}
\end{equation}
of degree 1 such that the following diagram commutes.

\begin{equation}
\label{eqn:commdig}
\begin{tikzcd}
  Z \arrow[swap]{d}{p} \arrow{r}{i_Z} &  \Jac(Y)^{\dual} \arrow{d}{\pi} \\
  X \arrow[swap]{r}{i_X} & \Kum(Y)^{\dual} \subset {\mathbb{P}}^3_k
\end{tikzcd}
\end{equation}
Furthermore, there exists a plane $H\subset \mathbb{P}^3_{k}$ such that $i_X(X) = H\cap\Kum(Y)^{\dual}$ and $H$ contains exactly two singular points of $\Kum(Y)^{\dual}$.
\end{theorem}
\begin{proof}
Using \cite[1.4]{barth} we see that the map $p$ induces a commutative diagram
\begin{equation}
\begin{tikzcd}
 & & 0 \arrow{d} & &  \\
& & \ker(\Nm) \arrow{d} \arrow{dr}{\rho} & & \\
0 \arrow{r} & \Jac(X) \arrow{r}{p^*} & \Jac(Z) \arrow{r} \arrow{d}{\Nm} & \Pr(Z/X)^{\dual}\arrow{r} & 0 \\
& Z \arrow{ur} \arrow[swap]{r}{p} & X \arrow{d} & & \\
& & 0 & &
\end{tikzcd}.
\end{equation}
where
\begin{enumerate}
\item $\ker(\Nm)^0 = \Pr(Z/X)$ is the Prym variety of $p$,
\item $\Nm$ is the norm map and $\Nm \circ p^*$ is multiplication by 2.
\item The map $j:Z\to \Pr(Z/X)^{\dual}$  is an embedding.
\item $\rho$ restricted to $\ker(\Nm)^0\to \Pr(Z/X)^{\dual}$
has 4 elements in its kernel and induces a (1,2)-polarization on $\Pr(Z/X)$, which is given as a divisor by $j(Z)$.
\item The involution $\iota$ on $Z$ defining the cover $p$ extends to multiplication by $-1$ on $\Pr(Z/X)^{\dual}$.
\end{enumerate}

 By \cite[Lemma 3.3]{ritzenthaler-romagny} we know that $\Pr(Z'/X')$ is isomorphic to $\Jac(Y)$ for a certain choice of a singular degree 2 cover $p':Z'\to X'$ whose desingularization is equal to the original cover $p:Z\to X$.
We furthermore remark that both $X'$ and $Z'$ have exactly two singular points, which are simple nodes and that these singular points coincide with the ramification points of the cover. Now using \cite[Lemma 1]{donagi-livne} we get a map  $\nu^*:\Pr(Z'/X')\to \Pr(Z/X)$ that sends the (1,2)-polarization to the principal polarization on $\Pr(Z/X)$. Taking the dual of $\nu^*$, we get a map
 $(\nu^*)^{\dual}:\Pr(Z/X)^{\dual} \to \Pr(Z'/X')^{\dual} = \Jac(Y)^{\dual}.$ Now define
 \[
   i_Z: Z \to \Jac(Y)^{\dual}
 \]
 by $i_Z = (\nu^*)^{\dual} \circ j$. As $j$ is an embedding by (iii) and $(\nu^*)^{\dual}$ is induced by the desingularization $Z\to Z'$ we get that $i_Z(Z)$ is isomorphic to $Z'$. As $\iota$ extends to multiplication by $-1$ on $\Pr(Z/X)^{\dual}$ by (v), it naturally extends to multiplication by $-1$ on $\Jac(Y)^{\dual}$.  The singular points of $i_Z(Z)$ pass through the points that are invariant under this involution (as the singular points are the ramification points), so the singular points of $Z'$ pass through the 2-torsion points of $\Jac(Y)^{\dual}$. Now $j(Z)$ induces a $(1,2)$-polarization on $\Pr(Z/X)^{\dual}$ by (iv). Since $\nu^*$ sends the $(1,2)$-polarization to a principal polarization, then $(\nu^*)^{\dual}$ sends $j(Z)$ to a divisor that is linearly equivalent to $2\Theta_Y$, where $\Theta_Y$ is the image of a theta divisor of $Y$ under the principal polarization $\lambda$. By \cite[Proposition 4.17]{Gonzalez-Dorrego} the map $\Jac(Y)^{\dual}\to \Kum(Y)^{\dual}$ is given by the linear system $|2\Theta_Y|$. This shows that $\pi(i_Z(Z))$ is given by $H\cap \Kum(Y)^{\dual}$ for some plane $H$. As $\pi$ commutes with $p$ by construction there exists a map $i_X:X\to \Kum(Y)^{\dual}$ such that the claim follows.
\end{proof}

\subsection{Geometric gluing data}
Here we will give a geometric interpretation of the gluing data given in Corollary \ref{cor:roottrans}.
\begin{lemma}
\label{lem:specialconic}
Let $\Kum(Y)$ be a Kummer surface and let $H$ be a special plane. Then $H\cap \Kum(Y)=2C$ where $C$ is a conic. These conics, called \defi{special conics}, are the only conics on $\Kum(Y)$.
\end{lemma}
\begin{proof}
See \cite[Theorem 2.6]{Gonzalez-Dorrego} and \cite[Proposition 2.18]{Gonzalez-Dorrego2}.
\end{proof}

\begin{proposition}
  Let $\Kum(Y) = \Jac(Y)/\langle -1\rangle\subset \PP^3_{k}$ be the Kummer surface associated to $\Jac(Y)$ and let $\pi:\Jac(Y) \to \Kum(Y)$ be the quotient map. Fix a Weierstrass point $Q$ on $\Jac(Y)$ and let $\Theta_Y$ be the theta divisor given by the image of the map $Y\to \Jac(Y),  P\mapsto P-Q$.
\begin{enumerate}
  \item
    Given $T \in \Jac(Y)[2]$, then $\pi(t_T^*(\Theta_Y))$ is a special conic on $\Kum(Y)$.
  \item
    Given a special conic $C$, then $C = \pi(t_T^*(\Theta_Y))$ for some $T \in \Jac(Y)[2]$. Thus $\pi^*(C) = t_T^*(\Theta_Y)$, i.e., $\pi^*(C)$ is a translate of $\Theta_Y$ by a 2-torsion point on $\Jac(Y)$.
\end{enumerate}
\end{proposition}
\begin{proof}
  (i) Let $Y_i$ be the curve given by $t^*_{T_i}(\Theta_Y)$ where $T_i\in \Jac(Y)[2]$. The curve $\Theta_Y$ intersects  $\Jac(Y)[2]$ in exactly six points, so it follows that the curves $Y_i$ also have this property as they are translates by a 2-torsion point. As $\pi$ is a degree 2 map that is ramified only above the 2-torsion points, it follows by Riemann-Hurwitz that $\pi(Y_i)$ has genus $0$. Since $\Kum(Y)$ contains no lines (cf., \cite[Theorem 4.30]{Gonzalez-Dorrego}), then $\pi(Y_i)$ must be a smooth conic on $\Kum(Y)$. By Lemma \ref{lem:specialconic} the special conics are the only conics on $\Kum(Y)$, so $\pi(Y_i)$ has to be a special conic $C$. It follows that $\pi^*(C) = Y_i$, as desired.

  (ii) Let $S_1,\ldots S_6$ be the six 2-torsion points that lie on $\Theta_Y$. Then $Y_i$ passes through the points $S_1+T_i, \ldots, S_6+T_i$. A calculation shows that these sets of six 2-torsion points are distinct for all sixteen $T_i\in \Jac(Y)[2]$. It follows that the curves $\pi(Y_i)$ are sixteen distinct conics on $\Kum(Y)$. As there are only sixteen special conics on $\Kum(Y)$, the claim follows and we have proven (ii).
\end{proof}

\begin{lemma}
  \label{lem:xproperties}
  Let $\Kum(Y)$ be as above and let $H$ be a plane in $\PP^3_k$ that intersects $\Kum(Y)$ in exactly two singular points $P_1,P_2$. Let $X' = H\cap \Kum(Y)$. Then
  \begin{enumerate}
    \item The curve $X'$ is a singular curve of genus 1 with exactly two nodes in $P_1$ and $P_2$.
    \item There are exactly four distinct lines $L_i$ in $H$ such that $P_1\in L_i, P_2\notin L_i$ and $X'$ is tangent to $L_i$. Furthermore, each of these $L_i$ is the intersection of $H$ with a special plane $H_i$ that contains $P_1$, but does not contain $P_2$.
  \end{enumerate}
\end{lemma}

\begin{proof}
  (i) follows from \cite[Proposition 2.20]{Gonzalez-Dorrego2}.

  (ii) Any line through $P_1$ that is tangent to $\Kum(Y)$ has to be contained in the intersection of the enveloping cone at $P_1$ and $H$. According to \cite[Theorem 2.6]{Gonzalez-Dorrego} the enveloping cone of $\Kum(Y)$ at $P_1$ consists of the six special planes of the Kummer surface that contain $P_1$. Using the incidence relations of the (16,6) configuration \cite[Lemma 1.7]{Gonzalez-Dorrego} we see that there are exactly four special planes $H_i$ containing $P_1$, but not $P_2$.  As $H$ itself cannot be a special plane (as it contains only two singular points) it follows that the intersections $L_i = H\cap H_i$ give us four distinct lines that have the required properties.
\end{proof}

\begin{remark}
  The incidence relations also tell us there are only two special planes that contain both $P_1$ and $P_2$.
\end{remark}

\begin{lemma}
  \label{lem:CRFuncadapt}
  Let  $\Kum(Y) \subset \mathbb{P}^3_k$ be a Kummer surface with singular points $P_i$. Let $H$ be a plane that contains exactly two singular points $P_1$ and $P_2$. Let $U\cong \mathbb{A}^2_k$ be an affine open of $H$ containing both $P_1$ and $P_2$.  Let $\widetilde{C} = \Kum(Y) \cap U$.  Let $(x_i,y_i)$ be the coordinates of $P_i$ in $U$, and define the function $\widetilde{g}:\widetilde{C}\backslash\{P_i\}\to k$ by
  \begin{equation}
    \widetilde{g}(x,y) = \frac{y-y_i}{x-x_i} \,.
  \end{equation}
  Then $\widetilde{g}$ extends to a function
  \begin{equation}
    g:C \to \mathbb{P}^1_k
  \end{equation}
  of degree 2 where $C$ is the normalization of $\widetilde{C}$. Furthermore, the ramification points of $g$ coincide with the slopes of the four lines $L_i$ from Lemma \ref{lem:xproperties}.
\end{lemma}

\begin{proof}
   Given a point $Q_1 = (x,y) \in \widetilde{C}$, then $\widetilde{g}(x,y)$ is exactly the slope of the line $\ell$ through $P_1$ and $Q_1$. As $\Kum(Y)$ is a quartic surface and $U$ is a plane, then $\Kum(Y) \cap U$ is a quartic plane curve. Then $\ell$ and $\widetilde{C}$ have intersection number $4$ by B\'{e}zout's theorem. Since $P_i$ is a node it contributes $2$ to the intersection number, so $\ell$ generically intersects $\widetilde{C}$ in a third point $Q_2$. Thus $\widetilde{g}$ is generically $2$-to-$1$, hence has degree $2$. The function $\widetilde{g}$ ramifies exactly when the the line $\ell$ is tangent to $\Kum(Y)^{\dual}$, which occurs exactly when $\ell$ is the intersection of $H$ with a special plane containing $P_1$. As described in the proof of Lemma \ref{lem:xproperties}, exactly four of these do not contain $P_2$ and hence will remain ramification points of the map $g:C\to \mathbb{P}^1_k$ on the normalization. As $C$ is a smooth genus $1$ curve by \ref{lem:xproperties}(i) then the map $g$ can have at most 4 ramification points. It follows that the ramification points of $g$ coincide with the slopes of the four lines $L_i$.
 \end{proof}

\begin{remark}
  We will later use the points in the reduced subscheme of the intersection of $X'$ with a special plane to describe the 2-torsion points on $\Jac(X')$.
\end{remark}

\begin{remark} \label{rem:dual_nonsense}
  Below we sometimes apply results stated for $\Jac(Y)$ (resp., $\Kum(Y)$) to $\Jac(Y)^{\dual}$ (resp., $\Kum(Y)^{\dual}$). This is permissible since $\Jac(Y)$ is principally polarized: we have an isomorphism $\Jac(Y) \overset{\sim}{\to} \Jac(Y)^{\dual}$ that also induces an isomorphism $\Kum(Y) \overset{\sim}{\to} \Kum(Y)^{\dual} = \Jac(Y)^{\dual}/\langle -1 \rangle$.
\end{remark}

\begin{lemma} Let $X'$ be as above and let $X$ be the normalization of $X'$. Assume that $\Jac(Y)^{\dual}$ is not isogenous to $\Jac(X) \times \Jac(E)$ for some elliptic curve $E$. Then $Z' = \pi^{-1}(X')$ is an irreducible singular curve of genus 3 with exactly two nodes.
\end{lemma}
\begin{proof}
  Assume that $Z'$ is reducible and write $Z' =C' \cup D'$ for two curves $C'$ and $D'$. As $\pi$ is of degree 2, both $C'$ and $D'$ must be birational to $X'$. Since $C'$ is birational to $X'$, then $X$ is the normalization of $C'$, and hence the inclusion $C'\to \Jac(Y)^{\dual}$ induces a surjective map $\Jac(Y) \to \Jac(X)$ (where we identify $(\Jac(Y)^{\dual})^{\dual}$ with $\Jac(Y)$). This contradicts the assumption that $\Jac(Y)^{\dual}$ is not isogenous to $\Jac(X)\times \Jac(E)$ for any $E$. When $Z'$ is irreducible, it follows by \cite[Proposition 2.20]{Gonzalez-Dorrego2} that it is a singular curve of genus 3 with exactly two nodes.
\end{proof}

Let $Z'$ be a singular genus 3 curve whose singular points consist of exactly two nodes. Let $\nu: Z\to Z'$ be its desingularization. Then the pullback map $\nu^*$ induces an exact sequence

\begin{equation} \label{eqn:singular_jacobian}
  0 \to (k^*)^2 \to \Jac(Z') \overset{\nu^*}{\to} \Jac(Z) \to 0
\end{equation}
(cf., \cite[Lemma 5.18]{Liu}).

\begin{proposition}\label{prop:kerinS}
  Let $\pi:\Jac(Y)^{\dual} \to \Kum(Y)^{\dual}$, $X'$, $H$ and $P_1$ be as in Lemma \ref{lem:xproperties}. Let $S_1\in \Jac(Y)^{\dual}$ be the point that has $\pi(S_1) = P_1$. Let $Z'\subset \pi^{-1}(X')$ be an irreducible curve of genus 3. Let $C_1, \ldots C_4$ be the four special conics that pass through $P_1$ and intersect $X'$. Let $Y_j = \pi^{-1} (C_j)$, fix $\Theta_Y = Y_1$, and identify $T_j\in (\Jac(Y)^{\dual})^{\dual}[2]= \Jac(Y)[2]$ with the divisor $Y_j - Y_1$ via the Abel-Jacobi map $P\mapsto t^*_P(Y_1)-Y_1$. Let $R_j$ be the point in $(C_j\cap X')_{\red}$ that is not equal to $P_1$. Then there exists an isogeny $\phi:\Jac(X)\times \Jac(Y)\to \Jac(Z')$ such that $\phi(\langle (R_j -R_1, T_j-T_1) \rangle)\subset \ker(\nu^*).$
\end{proposition}

\begin{proof}
  Let $p':Z'\to X'$ be the restriction of $\pi$ to $Z'$. We now consider the pullback map $p'^*:\Jac(X')\to \Jac(Z')$. The morphism $\pi$ is smooth outside of the 2-torsion points on $\Jac(Y)^{\dual}$ and we can calculate $p'^*$ by taking pre-images of points on the smooth locus. We can calculate the image of the 2-torsion of $\Jac(X')$ by studying the pre-image of the divisors $R_j-R_1$. Let $\iota_j$ be the involution on $Y_j$. By construction $\pi^*(R_j - R_1) = Q_j +\iota_j(Q_j) - Q_1 - \iota_j(Q_1)$ for some points $Q_j$ on $Y_j$.

  On the other hand let $i_{Z'}:Z'\to \Jac(Y)^{\dual}$ be the inclusion map. This induces a pullback map of divisors $i_{Z'}^*:\Jac(Y) \to \Jac(Z')$. The divisor $Y_j-Y_1$ corresponds via the polarization $P \mapsto t^*_P(Y_1) - Y_1$ to a 2-torsion point $T_j$ on $\Jac(Y)$.

  As $C_j = \pi(Y_j)$ is tangent to $X'$ in $R_j$ and intersects the node $P_1$, it follows that $Z'\cap Y_j = Q_j+\iota_j(Q_j)+2S_1$ if we purely look at the underlying sets. Outside of the singular locus we may treat Cartier divisors as if they are Weil divisors. When we pull back the divisor $T_j\in \Jac(Y)$ to $\Jac(Z')$ by $i_{Z'}$ however, we may need to take the contribution of the singular point $S_1$ into account.
  It now follows that
\begin{align}
  \begin{split}
  i_{Z'}^*(T_j-T_1) &\sim Q_j+\iota_j(Q_j)+s_1 -  Q_1 - \iota_1(Q_1) -s_j
  \end{split}
\end{align}
where the $s_i$ have the property that $\nu^*(s_i) = W_{1,1}+W_{1,2}$ where the $W_{1,i}$ are the two points above $S_1$ in the normalization $\nu:Z\to Z'$. Consider the map $p'^*+i_{Z'}^*:\Jac(X')\times \Jac(Y)\to \Jac(Z')$. By the above we find that
  \begin{equation}
    (p'^*+i_{Z'}^*) (R_j -R_1, T_j-T_1) = 2(Q_j +\iota_j(Q_j) - Q_1 - \iota_1(Q_1))+s_j-s_1 \sim  s_j-s_1
  \end{equation}
  as $Q_j +\iota_j(Q_j) - Q_1 - \iota_1(Q_1)$, being the image of a 2-torsion point, is a 2-torsion point. Furthermore,
  \[\nu^*(s_j-s_1) = W_{1,1}+W_{1,2}-W_{1,1}-W_{1,2}\sim 0,\] so we conclude that
  \begin{equation}
    (p'^*+i_{Z'}^*)(\langle (R_j -R_1, T_j-T_1) \rangle) \subset \ker \nu^*. \qedhere
  \end{equation}
\end{proof}

\begin{lemma}
  \label{lem:beta56}
  Let $i_{Z'}:Z'\to \Jac(Y)^{\dual}$ be as above and write $P_1$ and $P_2$ for the nodes of $Z'$. Let $T_5$ and $T_6$ be the two points on $\Jac(Y)$ such that $T_5 = Y_5 - Y_1$, $T_6 = Y_6-Y_1$ and $Y_5$ and $Y_6$ are the two curves that intersect both $P_1$ and $P_2$.  Then $i_{Z'}^*(T_5-T_6)\subset \ker \nu^*$.
\end{lemma}

\begin{proof}
  As points, we find that $Y_5 \cap Z' = Y_6\cap Z' = 2P_1+2P_2$ by construction. Similarly to the above proposition we have
  \begin{equation}
    i^*(T_5-T_6) = s_1 - s_2
  \end{equation}
  for some $s_1,s_2\in \Jac(Z')$ with the property that $\nu^*(s_i) = W_{1,1}+W_{1,2}+W_{2,1}+W_{2,2}$. (Similarly to what we did above the $W_{2,i}$ are the two points above $S_2$ in the normalization $\nu:Z\to Z'$). It follows that  $i_{Z'}^*(T_5-T_6)\subset \ker \nu^*$.
\end{proof}

We write $p: Z \to X$ for the normalization of the singular cover $p': Z'\to X'$.

\begin{proposition}
\label{prop:gluingkernel}
The kernel of the gluing isogeny $\phi:\Jac(X)\times \Jac(Y)\to \Jac(Z)$ where $\phi = (p^* + (i_{Z'}\circ \nu)^*)$ is given by
\begin{equation}\langle (R_j -R_1, T_j-T_1),(0,T_5-T_6)) \rangle.\end{equation}
\end{proposition}
\begin{proof}
  As we may assume that $\Theta_X$ does not lie on the singular part of $X'$,
  we may assume that
  \begin{equation}
    \ker \phi = \ker (p^* + (i_{Z'}\circ \nu)^*) = \ker \nu^*\circ(p'^* + i_{Z'}^*)
  \end{equation}
  Using Lemma \ref{prop:kerinS} and Lemma \ref{lem:beta56}  we find that
  \begin{equation}
    \langle (R_j -R_1, T_j-T_1),(0,T_5-T_6)) \rangle\subset \ker \phi.
  \end{equation}
  Since $\ker(\phi)$ is a maximal isotropic subgroup of $\Jac(X)[2]\times \Jac(Y)^{\dual}[2]$, then $\ker \phi$ is isomorphic to $\FF_2^3$ as a vector space. As $(R_2-R_1)$ and $(R_3-R_1)$ are linearly independent in $\Jac(X)[2]$, we see that $(R_2 -R_1, T_2-T_1), (R_3 -R_1, T_3-T_1)$ and $(0,T_5-T_6)$ are linearly independent in $\Jac(X)[2]\times \Jac(Y)[2]$. This completes the proof.
\end{proof}

By Theorem \ref{thm:embedding} we know that every non-hyperelliptic genus 3 gluing $Z$ can be mapped into the  $\Jac(Y)^{\dual}$ via an immersion and Proposition \ref{prop:gluingkernel} tells us how we can read off the gluing kernel. We can use this to see in which way the gluings show up.

\begin{lemma} \label{lem:jlambda}
  Let $P_1$ and $P_2$ be nodes of $\Kum(Y)$ and let $H(\lambda)$ be the family of planes going through $P_1$ and  $P_2$. Then the $j$-invariant of the family $H(\lambda)\cap \Kum(Y)$ is a rational function $j(H(\lambda))\in k(\lambda)$ of degree at most 12.
\end{lemma}

\begin{proof}
By \cite[Proposition 2.20]{Gonzalez-Dorrego} we may assume that $K$ is given by the homogeneous polynomial

\begin{equation} \kappa(x,y,z,t) = x^4 + y^4 + z^4 + t^4 + 2Dxyzt + A(x^2t^2 + y^2z^2) + B(y^2t^2 + x^2z^2) + C(z^2t^2 + x^2y^2)\end{equation}
in $\mathbb{P}^3_k$ with singular points $P_1 = (d,-c,b,-a)$ and $P_2 = (d,c,-b-a)$.  In this case the family of planes going through $P_1$ and $P_2$ is given by
\begin{equation}
H(\lambda) =  ax+by+cz+dt+\lambda(ax-by-cz+dt).
\end{equation}

Without loss of generality we assume that $b,d\neq 0$. Let $U$ be the affine open subset of $H_{1,2}(\lambda)$ that we get by setting $z=1$ to get a plane that contains both $P_1$ and $P_2$. Let $\widetilde{C}_{\lambda_0} = U\cap K$. It follows that we can describe $\widetilde{C}_{\lambda_0}$ as a curve in $\mathbb{A}^2_k$ given by the equation $F(x,y)=0$ where
\begin{equation}
F(x,y) = \kappa\left(x,y,1,\frac{(1+\lambda)ax+(1-\lambda)(by+c)}{d(-1-\lambda)}\right)
\end{equation}
 and define an isomorphism $\phi: \widetilde{C}_{\lambda_0}\to K\cap U$ by
\begin{equation}
\phi(x,y) = \left(x,y,1,\frac{(1+\lambda)ax+(1-\lambda)(by+c)}{d(-1-\lambda)}\right).\end{equation}
Using this isomorphism we get $\phi^{-1}(d,-c,b,-a)=(d/b,-c/b)$. Let
\begin{equation}g:(U\cap K)\backslash \{(d/b,-c/b)\}\to k\end{equation}
be the function defined by mapping a point $P$ to the slope of the line passing through $(d/b,-c/b)$ and $P$ as in Lemma \ref{lem:CRFuncadapt}. We will find the ramification points of $g$ to calculate the $j$-invariant of the family.

A line in $U$ with slope $\mu$ passing through $(d/b,-c/b)$ satisfies the equation
\begin{equation}y = \mu x -c/b-\mu d/b.\end{equation}
Consider the polynomial
\begin{equation}
F(x,\mu x -c/b-\mu d/b)
\end{equation}
in $k(\lambda,\mu)[x]$.

Let $D(\mu)\in k(\lambda)$ be the discriminant of $F/((x-d/b)^2)$ with respect to $x$. Solving $D(\mu)=0$ gives us the values of $\mu$ for which the intersection number of $L$ with $\widetilde{C}_{\lambda_0}$ is greater than 2. We divide by $(x-d/b)^2$ to exclude the case where $L$ intersects $P_1$.

A calculation shows that the zeroes of $D(\mu)$ are:
\begin{align}
0,\\
x_1(\lambda) &= ((ab\lambda + ab + cd\lambda + cd)/(b^2\lambda - b^2 - d^2\lambda - d^2)),\\
x_2(\lambda) &= ((ab\lambda + ab - cd\lambda - cd)/(b^2\lambda - b^2 + d^2\lambda + d^2)),\\
x_3(\lambda) &= ((-ac\lambda - ac - bd\lambda - bd)/(ad\lambda + ad - bc\lambda + bc)),\\
x_4(\lambda) &= ((-ac\lambda - ac + bd\lambda + bd)/(ad\lambda + ad - bc\lambda + bc)).
\end{align}
The 0 coincides with the horizontal line that passes through $P_2$. The other values give us the branch points of the map $g$.
Note that the $x_i(\lambda)$ are rational functions of degree 1 in $\lambda$. To compute the $j$-invariant of the normalization of $\widetilde{C}_{\lambda_0}$ we compute the cross-ratio of the $x_i(\lambda)$
\begin{equation}
\label{eq:crossratio}
c(\lambda) = \frac{(x_3-x_1)(x_4-x_2)}{(x_3-x_2)(x_4-x_1)} \end{equation}
which is a rational function of degree at most 2. It follows that the $j$-invariant
\begin{equation}j(\lambda) = \frac{(c(\lambda)^2-c(\lambda)+1)^3}{c(\lambda)^2(c(\lambda)-1)^2}.\end{equation}
is a rational function in $k(\lambda)$ of degree at most 12.
\end{proof}

\begin{proposition}
  Let $X$ and $Y$ be as above. Let $G$ be a maximal isotropic subgroup of $\Jac(X)[2]\times \Jac(Y)[2]$ and $(\mathcal{T},\ell)$ be the corresponding pair given by Proposition \ref{prop:class}. Then choosing $\mathcal{T}$ coincides with choosing two curves $Y_5$ and $Y_6$ on $\Jac(Y)^{\dual}$. Write $P_1$ and $P_2$ for the two singular points in $\pi(Y_5\cap Y_6)$. Assuming we have done this,  choosing $\ell$ is the same as picking a value of $\lambda_0$ in the family of planes $H_{1,2}(\lambda)$ such that $j(H_{1,2})(\lambda_0) = j(X)$. Moreover, for every choice of $\ell$ there generically exist two choices of $\lambda_0$ that have this property.
\end{proposition}
\begin{proof}
Fix a 2-torsion point $S_1$ on $\Jac(Y)^{\dual}$. By the (16,6)-configuration there exist distinct divisors $T_i$ with $i=1,\ldots 6$ such that all $Y_j$ (where $T_j = Y_j-Y_1$) contain the point $S_1$. Remark that as the $T_i$ are distinct, we have that $\langle T_i-T_j : i,j\in \{1,\ldots 6\}\rangle = \Jac(Y)[2].$ Now let $G$ be a maximal isotropic subgroup with corresponding pair $(\mathcal{T},\ell)$ and assume that $\mathcal{T} = \{T_5,T_6\}$. It follows from the incidence relations of the $(16,6)$-configuration that $Y_5\cap Y_6$ consists of exactly two points. One of them is the $S_1$ we defined before and we will call the other one  $S_2$.

Consider the sets
\begin{equation}\mathcal{P} = \left\{R_i\in \Jac(X)[2]: i\in\{1,\ldots, 4\}  \right\},\quad \mathcal{Q} = \left\{T_i\in \Jac(Y): i\in\{1,\ldots, 6\}  \right\}.\end{equation}
We remark that, analogously to Corollary \ref{cor:roottrans}, giving an indecomposable maximal isotropic subgroup of $\Jac(X)[2]\times\Jac(Y)[2]$ is the same as giving a subset $\mathcal{T}$ of cardinality 2 along with a symplectic morphism $\ell:\mathcal{G(P)}\to \mathcal{G(R)}$ where $\mathcal{R} = \mathcal{Q}\backslash \mathcal{T}.$

Now define $P_i = \pi(S_i)$ for $i=1,\ldots 2$ and let $H_{1,2}(\lambda)$ be the 1-dimensional family of planes in $\P^3$ parametrized by $\lambda$ that contains both $P_1$ and $P_2$. Using Lemma \ref{lem:jlambda} we see that there exist 12 values $\lambda_m$ such that $H_{1,2}(\lambda_m) \cap \Kum(Y)^{\dual}$ for $m\in \{1,\ldots, 12\}$ has the same $j$-invariant as $X$.  To simplify notation we write $X'_m = H_{1,2}(\lambda_m)\cap \Kum(Y)^{\dual}$. Generically, all of the planes in $H_{1,2}(\lambda)$ will contain exactly two singular points, so we may assume that the $X'_m$ are genus 1 curves with exactly two nodes. Without loss of generality we may assume that $R_j$ is the point in $ \pi(Y_j)\cap X'_1$ that is not equal to $P_1$.

Now let $Z_1$ be the normalization of the genus 3 component in $\pi^{-1}(X'_1)$. Using Proposition \ref{prop:RRIsGlue} and Theorem \ref{thm:embedding} we see that $Z_1$ is the gluing of $X$ and $Y$ with gluing datum $(\mathcal{T}, \ell_1)$ where $\mathcal{T} =\{T_5,T_6\} $ and where $\ell_1$ is defined by $R_j \mapsto T_j$. Moreover, as the cross-ratio \eqref{eq:crossratio} is a function of degree 2, it will generically assume all six values that correspond to $X$ twice. Choosing a different value for the cross-ratio permutes the $R_j$. This implies that for a given map $\ell:\mathcal{G(P)}\to \mathcal{G(R)}$ there exists exactly two $m$ such that $\ell(T_j) = \pi(Y_j)\cap X'_m$ and such that $Z_m$ is the gluing of $X$ and $Y$ with gluing datum $(\{T_5,T_6\},\ell)$.
\end{proof}

\subsection{Making the construction explicit}
\label{par:explicit}
Using the work of Mumford \cite{mumford-tata} and Cantor \cite{Cantor} on Jacobians of hyperelliptic curves, as well as the work of M\"{u}ller \cite{Mueller} on Kummer surfaces, we will give explicit descriptions of the objects and maps used in Theorem \ref{thm:embedding}.

We first recall how a divisor of degree 2 on a genus $2$ curve can be represented in Mumford coordinates as a pair of polynomials.

\begin{proposition} \label{prop:mumford_coord}
Let $Y$ be the affine part of a smooth curve of genus $2$ over $k$ given by a Weierstrass equation $y^2 = f(x)$ in $\P_k^2$
. Then there exists a bijection between the sets
\begin{align*}
\mathcal{S} \colonequals \{\{(x_1, y_1), (x_2, y_2)\} \in \Sym^2(Y) \mid x_1 \neq x_2 \}
\end{align*}
and
\begin{align*}
  \mathcal{P} \colonequals \{(a(x), b(x)) \in k[x] \times k[x] \mid \text{$a$ is monic and separable, $\deg(a) =2$, $\deg(b) \leq 1$} \} \, .
\end{align*}
\end{proposition}

\begin{proof}
See \cite[Proposition 1.2]{mumford-tata} or \cite[\S2]{Cantor}.
\end{proof}

Given a pair $(a(x), b(x)) \in \mathcal{P}$, then $a(x) = x^2 + a_1 x + a_2, b(x) = b_1 x + b_2$ for some $a_1, a_2, b_1, b_2 \in k$. We can use these coefficients as coordinates on an open affine subset of $\Jac(Y)$ as described in the following proposition.
\begin{proposition}
\label{prop:abJ}
Let $Y$ be  given by the equation $y^2 = f(x)$ in $\mathbb{P}_k^2$. Let $g_1$ and $g_2$ be polynomials in $k[a_1,a_2,b_1,b_2]$ such that
\begin{equation}
g_1(a_1,a_2,b_1,b_2)x +g_0(a_1,a_2,b_1,b_2) \equiv b(x)^2-f(x) \mod a(x).
\end{equation}
Then the system of equations
\begin{align}
g_1(a_1,a_2,b_1,b_2) &= 0, \\
g_2(a_1,a_2,b_1,b_2) &= 0
\end{align}
describes an affine open subset $U$ of $\Jac(Y)$ in $\mathbb{A}^4_k$.
\end{proposition}
\begin{proof}
See Proposition 1.3 and Chapter IIIa, \S2 of \cite{mumford-tata}.
\end{proof}

\begin{proposition}
\label{prop:Kumeq}
Let $Y$ be a curve of genus 2 over a field $k$ given by the equation
\begin{equation}y^2 = f_0 + f_1x + f_2x^2 + f_3x^3 + f_4x^4 + f_5x^5 + f_6x^6\end{equation} in $\mathbb{A}_k^2$.
Suppose $P = (x_1,y_1)$ and $Q = (x_2,y_2)$ are two points on $Y$ and let \hbox{$P+Q\in U\subset \Jac(Y)$} where $U$ is as in Proposition \ref{prop:abJ}.
Let \begin{equation}
\begin{split}
\kappa_1 &= 1, \\
\kappa_2 &= x_1+x_2, \\
\kappa_3 &= x_1x_2, \\
\kappa_4 &= \frac{F_0(x_1,x_2)-2y_1y_2}{(x_1-x_2)^2},
\end{split}
\end{equation}
where
\begin{equation} F_0(x,y) = 2f_0 + f_1(x + y) + 2f_2(xy) + f_3(x + y)xy + 2f_4(xy)^2 + f_5(x + y)(xy)^2 + 2f_6(xy)^3.\end{equation}
Then we can define a map $\pi:U\to \Kum(Y)$ given by $(P,Q)\mapsto (\kappa_1:\kappa_2:\kappa_3:\kappa_4)$ such that $\pi$ is equal to the quotient morphism $\Jac(Y)\to \Kum(Y)$ restricted to $U$.

The functions $\kappa_1, \kappa_2,\kappa_3,\kappa_4$ satisfy the quartic equation
\begin{equation} \label{eqn:Kummer_quartic}
K(\kappa_1, \kappa_2,\kappa_3,\kappa_4) = K_2(\kappa_1, \kappa_2,\kappa_3)\kappa_4^2
+ K_1(\kappa_1, \kappa_2,\kappa_3)\kappa_4 + K_0(\kappa_1, \kappa_2,\kappa_3) = 0 \, ,
\end{equation}
where
\begin{fleqn}
\begin{equation}\begin{split}
  K_2(\kappa_1, \kappa_2,\kappa_3) & = \kappa_2^2- 4\kappa_1\kappa_3 \\
\end{split}
\end{equation}
\begin{equation}
  \begin{split}
    K_1(\kappa_1, \kappa_2,\kappa_3)  & = - 4\kappa_1^3f_0 - 2\kappa_1^2\kappa_2f_1 - 4\kappa_1^2\kappa_3f_2 - 2\kappa_1\kappa_2\kappa_3f_3 \\
&\quad - 4\kappa_1\kappa_3^2f_4 - 2\kappa_2\kappa_3^2f_5 - 4\kappa_3^3f_6\\
  \end{split}
\end{equation}
\begin{equation}
  \begin{split}
    K_0(\kappa_1, \kappa_2, \kappa_3) & = -4\kappa_1^4f_0f_2 + \kappa_1^4f_1^2 -4\kappa_1^3\kappa_2f_0f_3 - 2\kappa_1^3\kappa_3f_1f_3 \\&\quad
-4\kappa_1^2 \kappa_2^2f_0f_4
+ 4\kappa_1^2\kappa_2\kappa_3f_0f_5 -4\kappa_1^2\kappa_2\kappa_3f_1f_4 -4\kappa_1^2\kappa_3^2f_0f_6 \\&
\quad + 2\kappa_1^2\kappa_3^2f_1f_5 - 4\kappa_1^2\kappa_3^2f_2f_4 + \kappa_1^2\kappa_3^2f_3^2  -4\kappa_1\kappa_2^3f_0f_5 \\
&\quad + 8\kappa_1\kappa_2^2\kappa_3f_0f_6 - 4\kappa_1\kappa_2^2
\kappa_3f_1f_5 +  4\kappa_1\kappa_2\kappa_3^2f_1f_6\\
&\quad - 4\kappa_1\kappa_2\kappa_3^2f_2f_5 -2\kappa_1\kappa_3^3f_3f_5 - 4\kappa_2^4f_0f_6-4\kappa_2^3
\kappa_3f_1f_6 \\
&\quad - 4\kappa_2^2\kappa_3^2f_2f_6 - 4\kappa_2\kappa_3^3f_3f_6 - 4\kappa_3^4f_4f_6 + \kappa_3^4f_2 \, ,
  \end{split}
\end{equation}
\end{fleqn}
and equation (\ref{eqn:Kummer_quartic}) gives us a projective embedding of $\Kum(Y)$ in $\PP^3_k$.
\end{proposition}

\begin{proof}
See {\cite[\S2]{Mueller}}.
\end{proof}

\begin{corollary}
\label{cor:Kummap}
Let $U$ be an affine open subset of $\Jac(Y)$ in $\mathbb{A}_k^4 = k[a_1,a_2,b_1,b_2]$ given by the system of equations $g_1 = 0, g_2 = 0$ as in Proposition \ref{prop:abJ}. Then the map $U\to \Kum(Y)$ described in Proposition \ref{prop:Kumeq} can be explicitly described as \begin{equation}(a_1,a_2,b_1,b_2)\mapsto \left(1:-a_1:a_2:\frac{\widetilde{F}_0(-a_1,a_2) - 2(b_1^2a_2 - b_1b_2a_1 + b_2^2)}{a_1^2 - 4a_2}\right)\end{equation}
where \begin{equation}\widetilde{F}_0(x,y) = 2f_0 + f_1x + 2f_2y + f_3xy + 2f_4y^2 + f_5xy^2 + 2f_6y^3.\end{equation}
\end{corollary}

\begin{proof}
The correspondence described in Proposition \ref{prop:mumford_coord} sends a pair of points $\{P_1, P_2\} \in \mathcal{S}$ with $P_i = (x_i, y_i)$, $i = 1,2$,
to the pair of polynomials $(a(x), b(x))$ with
$$
a(x) = (x - x_1)(x - x_2) \qquad \text{and} \qquad b(x) = \frac{y_2 - y_1}{x_2 - x_1}(x - x_1) + y_1 \, .
$$
The result now follows from equating coefficients and then substituting these expressions into the formula for $F_0$ given in Proposition \ref{prop:Kumeq}.
\end{proof}

\begin{remark}
  The point $(0:0:0:1)$ is always a singular point on the projective embedding of $\Kum(Y)$ in $\mathbb{P}^3_k$ given by equation \eqref{eqn:Kummer_quartic}, as can be verified by a computation of partial derivatives.
\end{remark}

\begin{lemma} \label{lem:findingH}
Let $\phi: k(\Kum(Y)) \to k(\Jac(Y))$ be the inclusion of function fields induced by the morphism $\pi:\Jac(Y)\to \Kum(Y)$.
Then
\begin{enumerate}
\item There exist $\alpha_i, \beta_j \in k(a_1,a_2)$ such that
\begin{align}
b_1b_2 &= \alpha_1(a_1,a_2)+\alpha_2(a_1,a_2)b_1^2, \label{eqn:b1b2} \\
b_2^2 &=\beta_1(a_1,a_2)+\beta_2(a_1,a_2)b_1^2
.\end{align}
\item Let
\begin{align}
h \colonequals \frac{(\kappa_2^2 - 4\kappa_3)\kappa_4 - \widetilde{F}_0(\kappa_2,\kappa_3) + 2 \kappa_2 \alpha_1(-\kappa_2, \kappa_3) + 2 \beta_1(-\kappa_2, \kappa_3)}{-2 \kappa_3 - 2 \kappa_2 \alpha_2(-\kappa_2, \kappa_3) - 2 \beta_2(-\kappa_2, \kappa_3)}
\end{align}
Then $\phi(h) = b_1^2$.
\end{enumerate}
\end{lemma}

\begin{proof}
Note that the polynomials $g_1$ and $g_2$ of Proposition \ref{prop:abJ} can be computed by dividing $b(x)^2 - f(x) = (b_1 x + b_2)^2 - f(x)$ by $a(x)$ using polynomial long division, hence are contained in the subring $k[a_1, a_2][b_1^2, b_1 b_2, b_2^2]$. Considering the system of equations $g_1 = g_2 = 0$ as linear equations in $b_1^2, b_1 b_2, b_2^2$ over the field $k(a_1, a_2)$, we can solve for $b_1 b_2$ and $b_2^2$ in terms of $b_1^2$.

Recall from Corollary \ref{cor:Kummap} that on the affine open of $\Jac(Y)$ with coordinates $a_1, a_2, b_1, b_2$ we have $\phi(\kappa_2) = -a_1$, $\phi(\kappa_3) = a_2$, and
\begin{align*}
\phi(\kappa_4) &= \frac{\widetilde{F_0}(-a_1, a_2) - 2(b_1^2 a_2 - b_1 b_2 a_1 + b_2^2)}{a_1^2 - 4 a_2} \, .
\end{align*}
Solving for $b_1 b_2$ and $b_2^2$ in terms of $a_1, a_2$, and $b_1^2$ as in part (i), we can express $\phi(\kappa_4)$ as a function of $\phi(\kappa_2), \phi(\kappa_3)$, and $b_1^2$. A somewhat laborious computation then shows that
\begin{align*}
b_1^2 = \frac{(\phi(\kappa_2)^2 - 4\phi(\kappa_3))\phi(\kappa_4) - \widetilde{F}_0(\phi(\kappa_2),\phi(\kappa_3)) + 2 \varphi(\kappa_2) \alpha_1(-\phi(\kappa_2), \phi(\kappa_3)) + 2 \beta_1(-\phi(\kappa_2), \phi(\kappa_3))}{-2 \phi(\kappa_3) - 2 \phi(\kappa_2) \alpha_2(-\phi(\kappa_2), \phi(\kappa_3)) - 2 \beta_2(-\phi(\kappa_2), \phi(\kappa_3))} \, .
\end{align*}
Thus for $h$ as defined as in the statement of the lemma, we have $\phi(h) = b_1^2$.
\end{proof}

\begin{corollary} Let $\pi,\phi$ and $h$ be as in Lemma \ref{lem:findingH}. Then we can extend $\phi$ to a morphism $\overline{\phi}: k(\Kum(Y))[\sqrt{h}] \to k(\Jac(Y))$ such that $\overline{\phi}$ is an isomorphism. Furthermore, let $C$ be a curve on $\Kum(Y)$, let $k(C)$ be the function field of $C$, and suppose that $D = \pi^{-1}(C) \subset \Jac(Y)$ is irreducible. Then $k(C)[\sqrt{h}]$ is the function field of $D$.
\end{corollary}

\begin{proof}
Define $\overline{\phi}(f + g \sqrt{h}) = \phi(f) + \phi(g) b_1$. As $a_1, a_2, b_1, b_2$ are the coordinates of a dense affine open subset of $\Jac(Y)$, it suffices to show that they are in the image of $\overline{\phi}$. From Corollary \ref{cor:Kummap} we have $a_1 = -\overline{\phi}(\kappa_2)$, $a_2 = \overline{\phi}(\kappa_3)$ and $\overline{\phi}(\sqrt{h}) = b_1$, and from equation \eqref{eqn:b1b2} we have
$$
b_2 = \frac{\alpha_1(a_1, a_2)}{b_1} + \alpha_2(a_1,a_2) b_1
$$
so $b_2$ is also in the image of $\overline{\phi}$.

The final statement follows from the fact that the inclusion of function fields $k(C) \hookrightarrow k(C)[\sqrt{h}]$ corresponds to the morphism of curves $D \to C$.
\end{proof}

\begin{remark}
  The condition that $D = \pi^{-1} (C)$ be irreducible is satisfied in the context, as Theorem \ref{thm:embedding} shows.
\end{remark}

The following result is a slight generalization of \cite[Theorem 1.1]{enolski}.

\begin{proposition}
\label{prop:4coversexist}
Let $C$ be a genus 1 curve over an algebraically closed field $k$ with $\chr (k)\neq 2$, and let $P_1,P_2,P_3,P_4$ be distinct points in $C$. Then
\begin{enumerate}
  \item There are exactly four distinct covers $D_j\to C$ (where $j=1,\ldots 4$) of degree 2 that are ramified above the $P_i$ and unramified everywhere else.
  \item If there exists a function $f \in k(C)$ such that $k(D_i) \cong k(C)[\sqrt{f}]$ for some $i$, then for each $j = 1, \ldots, 4$ there exists a degree-$0$ divisor $T$ (representing a non-trivial element of $\Pic(C)[2]$) such that $D_j$ is isomorphic to a curve with function field $k(C)[\sqrt{f_T}]$ with $\Div(f_T) = \Div (f)+2T$.
  \item Let $T_1,\ldots, T_4$ be the order-2 Weierstrass points on $C$. There exist non-constant functions $u \in L (T_1 + T_i)$ and $v \in L (2 T_1 + 2 T_i)$ with $\Div(v) = P_1+P_2+P_3+P_4-2T_1-2T_i$ such that
  \begin{enumerate}
    \item The curve $C$ has an equation of the form
      \begin{equation} \label{eqn:genus1}
        v^2+vh(u)+f(u) =0
      \end{equation}
      where $h$ is a polynomial of degree 2 and $f$ is a polynomial of degree 4.
      \item The curve $D_i$ has an equation  over $k$ of the form
      \begin{equation} \label{eqn:genus3}
        t^4+t^2h(s)+f(s) =0
      \end{equation}
      where $h$ is a polynomial of degree 2 and $f$ is a polynomial of degree 4.
  \item The cover $\pi_i:D_i\to C$ is explicitly given by $\pi_i(s,t) = (s,t^2)$.
  \end{enumerate}
\end{enumerate}
\end{proposition}
\begin{proof}
See \cite[Paragraph 4.4]{hanselman-thesis}.
\end{proof}

Finally, we give an explicit expression for the map $i_Z$ from Theorem \ref{thm:embedding}.

\begin{proposition} \label{prop:explicit}
  For a non-hyperelliptic curve $Z$ of genus 3 that is a gluing of $X$ and $Y$ the code in \cite{s-prymmap} contains an explicit description of the map $i_Z: Z\dashrightarrow \Jac(Y)^{\dual}$ from Theorem \ref{thm:embedding}.
\end{proposition}
\begin{proof}
According to Proposition \ref{lem:gon} the gluing $Z$ gives rise to a cover as in Theorem \ref{thm:riro}.
The code in \cite{s-prymmap} contains an explicit rational map $Z\dashrightarrow\Jac(Y)^{\dual}$. After a change of coordinates we may assume that $Z$ has an affine open $V$ of the form
\begin{equation}
v^4 + v^2g(u) + uh(u)=0
\end{equation}
where $g(u)= g_2u^2 + g_1u + g_0$ and
$h(u) = h_2u^2 + h_1u + h_0$ . We calculate an equation for $Y$ using \ref{thm:riro} and use this equation to construct the affine open $U\subset k[a_1,a_2,b_1,b_2]$ of the Jacobian $\Jac(Y)^{\dual}$ given by the equations in Proposition \ref{prop:abJ}.

Let
\scriptsize
\begin{equation}
\begin{split}
\alpha(u,v) &= (g_2 h_0 - g_0 h_2)v^2 + (g_2^2h_0 - g_2g_0h_2)u^2 + (g_2g_1h_0 - g_2g_0h_1 - h_2h_0)u,\\
\beta(u,v) &= g_2^2h_0 - g_2g_0h_2v^3 + (g_2^3h_0 - g_2^2g_0h_2)u^2\\
 +& ((g_2^2g_1h_0 - g_2g_1g_0h_2 - g_2h_2h_0 + g_0h_2^2)u + g_2^2g_0h_0 - g_2g_0^2h_2)v\\
 N(u,v) &= (g_2^2h_1 - g_2g_1h_2 + h_2^2)u + g_2^2h_0 - g_2g_0h_2.
\end{split}
\end{equation}
\normalsize
Then the map $i_Z:V\to U$ is explicitly given by
\begin{equation}
(u,v)\mapsto (\alpha(u,v)/N(u,v),0,\beta(u,v)/(N(u,v)),\beta(u,v)/(uN(u,v))).
\end{equation}
In the code it is shown that the image of $i_Z$ is generically contained in $U$. A proof for the generic injectivity of the map $i_Z:Z\to \Jac(Y)^{\dual}$ due to Davide Lombardo is the following: Assume that $i_Z$ is not injective. If $i_Z(Z)$ is of genus 2, then $Z$ would be hyperelliptic by Proposition \ref{prop:hyp}, which gives us a contradiction, so $i_Z(Z)$ is either of genus 1 or of genus 0. It is impossible for $i_Z(Z)$ to be of genus 0, as then by the theory of abelian varieties the map $i_Z$ would be constant. On the other hand, if $i_Z(Z)$ is a curve of genus 1 then $\Jac(Y)^{\dual}$ would be isogenous to the product of two elliptic curves, which cannot be true generically.

Let $\pi: U\to \Kum(Y)^{\dual}$ be the map given in Corollary \ref{cor:Kummap}. As $i_Z(Z)$ is contained in the plane given by $a_2=0$, it follows that $\pi(i_Z(Z))$ is a curve  contained in the plane $H$ defined by $\kappa_3=0$. This yields a rational map $Z\dashrightarrow \pi(i_Z(Z))$ of degree 2 . We claim that the curve $\pi(i_Z(Z))$ is of genus 1. Indeed, if $\pi(i_Z(Z))$ is not of genus 1 then it will either be of genus 2 or of genus 0. But in both of these cases $Z$ would be a hyperelliptic curve. Indeed, if $\pi(i_Z(Z))$ has genus 2 then Proposition \ref{prop:hyp} tells us that $Z$ is a hyperelliptic curve. In the second case we have a degree 2 cover from $Z$ to a genus 0 curve, so the statement follows by definition. We therefore conclude that $\pi(i_Z(Z))$ is of genus 1. As any plane section of a quartic surface in $\mathbb{P}^3_k$ has arithmetic genus 3 this means that the plane $H$ has to intersect $\Kum(Y)^{\dual}$ in two singular points. Finally, it remains to be shown that the above diagram commutes. Let $i:Z\to Z$ be the involution $(u,v)\mapsto (u,-v)$ that corresponds to the degree 2 cover $Z\to X$. Then the explicit form of the maps at \cite{s-prymmap} shows that
\begin{equation}
\begin{split}
i_Z(i(u,v)) &= i_Z((u,-v))\\
 &= (\alpha(u,-v)/N(u,-v),0,\beta(u,-v)/(N(u,-v)),\beta(u,-v)/(-vN(u,-v)))\\
&=(\alpha(u,v)/N(u,v),0,-\beta(u,v)/(N(u,v)),-\beta(u,v)/(uN(u,v))).
\end{split}
\end{equation}

Now the map $(a_1,a_2,b_1,b_2)\mapsto (a_1,a_2,-b_1,-b_2)$ sends the divisor $P+Q$ corresponding to the equations $x^2+a_1x+a_2$ and $y = b_1x+b_2$ to the divisor $P'+Q'$ corresponding to the equations $x^2+a_1x+a_2$ and $y=-b_1x-b_2$. Therefore $i_Z\circ i$ is multiplication by $-1$ on $\Jac(Y)^{\dual}$ and we conclude that we have found a commutative diagram as in \eqref{eqn:commdig}.
\end{proof}

\subsection{The algorithm}
In this section we combine the previous results into an algorithm for computing all non-hyperelliptic gluings of a genus $1$ and a genus $2$ curve.

\begin{remark}
The following algorithms may require extensions of the base field in order to produce explicit equations.
\end{remark}

\begin{algorithm} \label{alg:AlgGl}
This algorithm constructs all gluings of $X$ and $Y$ that produce non-hyperelliptic curves.

\noindent
\emph{Input:} Curves $X:y^2 = p_X$ and $Y:y^2 = p_Y$ of genera 1 and 2, respectively, and a pair $(T_5-T_6,\ell)$ as in Proposition \ref{prop:class}.

\noindent
\emph{Output:} A degree 2 cover $p:Z\to X$ such that $Z$ is a gluing of $X$ and $Y$ with gluing datum $G$, where $G$ is the maximal isotropic subgroup corresponding to $(T_5-T_6,\ell)$.

\noindent
\emph{Steps:}
\begin{enumerate}[(1)]
\item Calculate an affine model for $\Jac(Y)^{\dual}$ as in Corollary \ref{prop:abJ}.
\item Compute $j(X)$.
\item Calculate a model for $\Kum(Y)^{\dual}$ and the projection map $\Jac(Y)^{\dual}\to \Kum(Y)^{\dual}$ as in Proposition \ref{prop:Kumeq}.
\item Calculate a function $h$ with the property that $k(\Kum(Y)^{\dual})[\sqrt{h}]\cong k(\Jac(Y)^{\dual})$ as in Lemma \ref{lem:findingH}.
\item Let $P_1$ be the singular point $(0:0:0:1)$ on $\Kum(Y)^{\dual}$ and let $\Theta_Y$ be a theta divisor such that $\pi(\Theta_Y)$ passes through $P_1$.
\item Determine $P_2$ such that $\pi(Y_5) = \pi(\Theta_Y)$ and $\pi(Y_6)$ intersect both $(0:0:0:1)$ and $P_2$.
\item Calculate the 1-dimensional family $H_{1,2}(\lambda)$ of planes that pass through $P_1$ and $P_2$.
\item Calculate the set $\Lambda(X)$ of all $\lambda$ such that $j(H_{1,2}(\lambda) \cap \Kum(Y)^{\dual}) = j(X)$.
	\item
	For each $\lambda_0\in \Lambda(X)$:
	\begin{enumerate}
	\item Calculate the points $R_j$ where $R_j \in \pi(Y_j)\cap X'$ and $R_j\neq P_1$.
      \item If $\ell(T_j-T_1)$ is linearly equivalent to the  divisor $R_j-R_1$ on $X'$ for all $j$.
      \begin{enumerate}
		\item Calculate the curve $Z$ with function field $k(\widetilde{X}(\lambda_0))[\sqrt{h}]$ using Algorithm \ref{alg:covers}; this gives us the desired gluing and a natural projection map to $X$.
		\item Return the calculated gluing $Z$.
	\end{enumerate}
\end{enumerate}
\end{enumerate}
\end{algorithm}

\begin{remark}
By the (16,6)-configuration we can uniquely identify the divisors $Y_j$ in the algorithm by choosing six 2-torsion points on $\Jac(Y)$.
\end{remark}

\begin{algorithm}
This algorithm calculates an equation for the curve $Z$ with function field $k(\widetilde{X}(\lambda_0))[\sqrt{h}]$.
\label{alg:covers}
\\
\emph{Input:}
\begin{itemize}
  \item A singular genus-1 curve $\widetilde{X}(\lambda_0)$ with exactly two nodes that is given by the intersection of a plane and a quartic surface.\\
\item A function $h\in k(\widetilde{X}(\lambda_0))$ whose divisor is of the form $P_1+P_2+P_3+P_4 - 2T$ for some divisor $T$.
\end{itemize}
\emph{Output:} A degree 2 cover $p:Z\to X(\lambda_0)$ such that $k(Z)\cong k(\widetilde{X}(\lambda_0))[\sqrt{h}]$ where $X(\lambda_0)$ is the normalization of $\widetilde{X}(\lambda_0)$.

\noindent
\emph{Steps:}
\begin{enumerate}[(1)]
\item Use Lemma \ref{lem:CRFuncadapt} to compute the branch points $\alpha_1,\ldots \alpha_4$ of the map $g:\widetilde{X}(\lambda_0)\to \mathbb{P}^1_k$ that maps a point $P$ to the slope of the line that connects $P$ with the singular point $(1:0:0:0)$.
\item Define the curve $C$ by the equation $y^2=(x-\alpha_1)(x-\alpha_2)(x-\alpha_3)(x-\alpha_4)$ and compute a birational map $\tau:C\to\widetilde{X}(\lambda_0)$.
\item Let $Q,R$ be the singular points on $\widetilde{X}(\lambda_0)$ and calculate $\tau^{-1}(Q) = \{Q_1,Q_2\}$ and $\tau^{-1}(R) = \{R_1,R_2\}$.
\item Calculate the divisor $D$ of the image of $h$ in $k(C)$.
\item Find divisors $D_1,D_2,D_3,D_4$ that correspond to the four distinct degree 2 covers with ramification points $Q_1,Q_2,R_1,R_2$ as in Proposition \ref{prop:4coversexist}.
\item For each {$i\in \{1,\ldots, 4\}$}:
\begin{enumerate}
\item If there exists a principal divisor $T$ such that $D_i-D =2T$, then:
\begin{enumerate}
\item Calculate the degree 2 cover $Z \to X$ corresponding to $D_i$ using Riemann-Roch spaces which is possible according to Proposition \ref{prop:4coversexist}(iii).
\item Return a quartic equation of $Z$ along with the map of degree 2 to $X$.
\end{enumerate}
\end{enumerate}
\end{enumerate}
\end{algorithm}

\begin{example}
  We illustrate the field extensions required by the second method. Let $X$ be the curve given by
\begin{equation}
y^2 = x^4 + 2x^3 - x^2 - 2x
\end{equation}
and let
$Y$ be the curve given by
\begin{equation}
y^2 = x^6 - 2x^5 - 10x^4 + 20x^3 + 9x^2 - 18x
\end{equation}
over $\QQ$. An affine open of $\Jac(Y)^{\dual}$ is given by the following system of equations in $\QQ[a_1,a_2,a_3,a_4]$.

\scriptsize
\begin{equation}
\begin{split}
-a_1^4a_2 &- 2a_1^3a_2 + 3a_1^2a_2^2 + 10a_1^2a_2 + 4a_1a_2^2 + 20a_1a_2 - a_2^3 -10a_2^2 + a_2b_1^2 - 9a_2 - b_2^2 =0,\\
-a_1^5 - 2a_1^4& + 4a_1^3a_2 + 10a_1^3 + 6a_1^2a_2 + 20a_1^2 - 3a_1a_2^2 - 20a_1a_2
+ a_1b_1^2 - 9a_1 - 2a_2^2 - 20a_2 - 2b_1b_2 - 18=0.
\end{split}
\end{equation}
\normalsize
The equation for $\Kum(Y)^{\dual}$ in $\PP_{\QQ}^3$ is
\scriptsize
\begin{equation}
\begin{split}
324x_1^4 &+ 720x_1^3x_3 - 720x_1^2x_2x_3 - 144x_1x_2^2x_3 + 72x_2^3x_3 +
    832x_1^2x_3^2 - 36x_2^2x_3^2
    + 80x_1x_3^3 - 80x_2x_3^3
    + 44x_3^4 +\\&
    36x_1^2x_2x_4 - 36x_1^2x_3x_4 - 40x_1x_2x_3x_4 + 40x_1x_3^2x_4 +
    4x_2x_3^2x_4 - 4x_3^3x_4 + x_2^2x_4^2 - 4x_1x_3x_4^2 =0
\end{split}
\end{equation}
\normalsize
and the morphism $\pi:\Jac(Y)^{\dual}\to\Kum(Y)^{\dual}$ is explicitly given by
\scriptsize
\begin{equation}
\begin{split}&\pi(a_1,a_2,b_1,b_2) = \\
&\left[1:-a_1:a_2:\frac{2a_1a_2^2 - 20a_1a_2 + 2a_1b_1b_2 + 18a_1 + 2a_2^3 - 20a_2^2 -
    2a_2b_1^2 + 18a_2 - 2b_2^2}{a_1^2 - 4a_2}\right].
    \end{split}
\end{equation}
\normalsize
We consider the family of planes $H_{1,2}(\lambda)$ passing through the singular points $P = (0:0:0:1)$ and $Q = (-1/6 : 1/3 : 1/2 : 1)$.
The $j$-invariant of $X$ is $35152/9$, and we seek to find the values of $\lambda$ such that $H_{1,2}(\lambda) \cap \Kum(Y)^{\dual}$ has the same $j$-invariant. We calculate the $j$-invariant $j(\lambda)$ of $H_{1,2}(\lambda) \cap \Kum(Y)$ and find that the numerator of $j(\lambda)-35152/9$ factors as
\scriptsize
\begin{equation} \label{eqn:lambda_roots}
\left(\lambda - \frac{9}{23}\right)
\left(\lambda - \frac{1}{11}\right)
\left(\lambda^2 - \frac{38}{67}\lambda - \frac{9}{67}\right)
\left(\lambda^2 - \frac{98}{193}\lambda - \frac{3}{193}\right)
\left(\lambda^2 - \frac{42}{85}\lambda + \frac{1}{85}\right)
\left(\lambda^2 - \frac{22}{47}\lambda + \frac{3}{47}\right)
\left(\lambda^2 - \frac{2}{5}\lambda + \frac{1}{5}\right).
\end{equation}
\normalsize

We will construct the degree 2 cover above $\widetilde{X}(9/23) = H_{P,Q}(9/23)\cap\Kum(Y)^{\dual}$. (The other covers can be computed in the same way by choosing other roots of \eqref{eqn:lambda_roots}.
To compute $k(\widetilde{X}_1(-3/2))/(t^2 - h)$ we proceed as in Algorithm \ref{alg:covers} and calculate the branch points of the degree 2 map $g:\widetilde{X}(9/23)\to \QQ$. We compute the Legendre model $y^2 = x(x-1)(x-1/4)$ for $\widetilde{X}(9/23)$, which we denote $\widetilde{X}_{\leg}$. Although $\widetilde{X}_{\leg}$ is a nontrivial twist of $\widetilde{X}(9/23)$ over $\QQ$, the curves become isomorphic upon extending scalars to $\QQ(\alpha)$, where $\alpha$ is a root of
$$
t^2 - 156026658225043557710221401/34308279913908709968852208000 \, .
$$
We explicitly compute the image of the function $h$ in the function field of $k(\widetilde{X}_{\leg}) \otimes \QQ(\alpha)$ and obtain a rational function of degree 14 with rather large coefficients, which we denote $h_{\leg}$.

Let $P_1, P_2$ (resp., $Q_1, Q_2$) be the two points obtained by desingularizing $P$ (resp., $Q$). We find that the divisor $P_1+P_2+Q_1+Q_2$ is defined over the field $\QQ(\beta,\gamma)$ where $\beta$ is a root of $t^2 + 3/32$ and $\gamma$ is a root of  $t^2 - 327/250t + 4761/10000$.
Let $T_1,\ldots T_4$ be the $2$-torsion points of $\widetilde{X}_{\leg}$. For each $i = 1, \ldots, 4$, we calculate functions $f_i$ with $\Div f_i = P_1+P_2+Q_1+Q_2 - 2T_i-2T_1$ for $i = 1, \ldots, 4$. We determine which $f_i$ corresponds to our covering $Z \to X$ by checking if $\Div (f_i/h_{\leg})$ is divisible by $2$ for each $i$. Applying Riemann-Roch as in Proposition \ref{prop:4coversexist} we compute the equation
\scriptsize
\begin{equation}\begin{split} u^4 &- \frac{244312307247680}{12491063134299}\alpha u^3 + \left(\frac{286830015625}{36438849216}\beta\gamma - \frac{250115773625}{48585132288}\beta\right)u^2v^2  + \frac{5876}{8855}u^2\\& +\left(\frac{-50500786167745625000}{1338579798660883737}\alpha\beta\gamma + \frac{11009171384568546250}{446193266220294579}\alpha\beta \right)uv^2 - \frac{83804221642880}{37473189402897}\alpha u \\ &-\frac{1044509681265625}{171408346712064}v^4 + \left(\frac{52518171875}{63767986128}\beta\gamma - \frac{45795845875}{85023981504}\beta \right)v^2 +\frac{1460}{111573}=0.
\end{split}
\end{equation}
\normalsize
for the gluing over $\QQ(\alpha,\beta,\gamma)$.
\normalsize
A simplified equation of this curve over $\QQ$, found by the methods in Section \ref{sec:interpolation}, is
\scriptsize
\begin{equation}
\begin{split}
12x^4 - 111x^2y^2 + 478x^2yz - 577x^2z^2 - 533y^4 + 948y^3z - 2574y^2z^2 + 2196yz^3 - 2277z^4=0.
\end{split}
\end{equation}
\normalsize
Over finite fields, the current Kummer method is less involved, since coefficient explosion is excluded and the required field extensions are less complicated.
\end{example}

\bibliographystyle{abbrv}
\bibliography{references}

\end{document}